\documentclass[12pt]{amsart}

\usepackage{graphicx}
\usepackage{color}
\usepackage{url}
\usepackage{dsfont}
\usepackage{extarrows}
\usepackage{amssymb}
\usepackage{capt-of}

\newtheorem{theorem}{Theorem}[section]
\newtheorem{lemma}[theorem]{Lemma}
\newtheorem{prop}[theorem]{Proposition}
\newtheorem{cor}[theorem]{Corollary}

\theoremstyle{definition}
\newtheorem{defi}[theorem]{Definition}
\newtheorem{example}[theorem]{Example}

\theoremstyle{remark}
\newtheorem{remark}[theorem]{Remark}

\numberwithin{equation}{section}

\newcommand{\rr}{{\mathbb R}}

\newcommand{\eqd}{\stackrel{d}{=}}

\newcommand{\diag}{\operatorname{diag}}

\newcommand{\plusmax}{\begin{picture}(10,5)
\put(1,0){$\vee$}
\put(0.5,5){$+$}
\end{picture}
}

\newcommand{\plusmin}{\begin{picture}(10,5)
\put(1,0){$\wedge$}
\put(0.5,5){$+$}
\end{picture}
}

\newcommand{\overbar}[1]{\mkern 1.5mu\overline{\mkern-1.5mu#1\mkern-1.5mu}\mkern 1.5mu}
\newcommand{\bunderline}[1]{\mkern 1.5mu\underline{\mkern-1.5mu#1\mkern-1.5mu}\mkern 1.5mu }
\newcommand{\Raum}{\mathbb{R}_{+} \times \bunderline{\overbar{\mathbb{R}}}}
\newcommand{\raum}{\mathbb{R}_{+}}
\newcommand{\Ru}{\bunderline{\mathbb{R}}}
\newcommand{\Rpmu}{\bunderline{\overbar{\mathbb{R}}}}
\newcommand{\Rplus}{\mathbb{R}_{+}}
\newcommand{\Rplushochzwei}{\mathbb{R}_{+}^2}
\newcommand{\Rkomp}{\overbar{\mathbb{R}}}
\newcommand{\ind}[1]{\mathds{1}_{#1}}

\newcommand{\N}{\mathbb{N}}
\newcommand{\Konvergenzninfty}{\xrightarrow[n \rightarrow \infty]{}}

\newcommand{\wKonvergenz}{\xlongrightarrow{w}}
\newcommand{\vageKonvergenz}{\xlongrightarrow{v'}}

\newcommand{\WKonvergenz}{\xLongrightarrow[n \rightarrow \infty]{}}
\newcommand{\wKonvergenzninfty}{\xrightarrow[n \rightarrow \infty]{w}}
\newcommand{\vKonvergenzninfty}{\xrightarrow[n \rightarrow \infty]{v'}}

\newcommand{\intnu}{\int_{0}^{\infty}}

\newcommand{\limn}{\lim_{n \rightarrow \infty}}
\newcommand{\WRaum}{\mathcal{M}^1(\Raum)}
\newcommand{\LargerCdot}{\raisebox{-0.25ex}{\scalebox{1.2}{$\cdot$}}}

\begin{document}

\title[Joint sum/max stability]{On joint sum/max stability and sum/max domains of attraction}
\author{Katharina Hees and Hans-Peter Scheffler} 
\address{Katharina Hees, Institut f\"ur medizinische Biometrie und Informatik, Universit\"at Heidelberg, 69126 Heidelberg, Germany} 
\email{hees@imbi-uni-heidelberg.de}
\address{Hans-Peter Scheffler, Department Mathematik, Universit\"at Siegen, 57072 Siegen, Germany} 
%\curraddr{...} 
\email{Scheffler@mathematik.uni-siegen.de}
%\urladdr{...} 
%\dedicatory{...} 
%\date{...} 
%\thanks{...} 
%\translator{...} 

\begin{abstract}
Let $(W_i,J_i)_{i \in \mathbb{N}}$ be a sequence of i.i.d. $[0,\infty) \times \rr$-valued random vectors. Considering the partial sum of the first component and the corresponding maximum of the second component, we are interested in the limit distributions that can be obtained under an appropriate scaling. In the case that $W_i$ and $J_i$ are independent, the joint distribution of the sum and the maximum is the product measure of the limit distributions of the two components. But if we allow dependence between the two components, this dependence can still appear in the limit, and we need a new theory to describe the possible limit distributions. This is achieved via harmonic analysis on semigroups, which can be utilized to characterize the scaling limit distributions and describe their domains of attraction. \end{abstract}
\maketitle

\section{Introduction}

Limit theorems for partial sums and partial maxima of sequences of i.i.d.\ random variables have a very long history and form the foundation of many applications of probability theory and statistics. The theories, but not the methods, in those two cases parallel each other in many ways. In both cases the class of possible limit distributions, that are sum-stable and max-stable laws, are well understood. Moreover, the characterization of domains of attraction is in both cases based on regular variation. See e.g. \cite{gnedenko, feller, thebook, resnick} to name a few. 

The methods used in the analysis in those two cases appear, at least at the first glance, to be completely different. In the sum case one usually uses Fourier- or Laplace transform methods, whereas in the max case the distribution function (CDF) is used. However, from a more abstract point of view these two methods are almost identical. They are both harmonic analysis methods on the plus resp.\ the max-semigroup. 

Surprisingly, a thorough analysis of the joint convergence of the sum and the maximum of i.i.d.\ random vectors, where the sum is taken in the first coordinate and the max in the second coordinate has never been carried out. Of course, if the components of the random vector are independent, one can use the classical theories componentwise and get joint convergence. To our knowledge, the only other case considered is the case of complete dependence where the components are identical, see \cite{chow1978sum}.

The purpose of this paper is to fill this gap in the literature and to present a  theory that solves this problem in complete generality. Moreover, there is a need for a general theory describing the dependence between the components of the limit distributions of sum/max stable laws. For example, in \cite{silvestrov} on page 1862 it is explicitly asked how to describe such limit distributions. Moreover, there are various stochastic process models and their limit theorems, that are constructed from the sum of non-negative random variables $W_i$, interpreted as waiting times between events of magnitude $J_i$, which may describe the jumps of a particle, in particular the {\it continuous time random maxima processes} studied in \cite{m3_stoev,shumer}, or the {\it shock models} studied in \cite{shan1,shan2,shan3,anderson,pancheva}. In those papers it is either assumed that the waiting times $W_i$ and the jumps $J_i$ are independent or asymptotically independent, meaning that the components of the limiting random vector are independent.

Motivated by these applications, in this paper we only consider the case of non-negative summands. More precisely, let
 $(W_i,J_i)_{i \in \mathbb{N}}$ be a sequence of i.i.d. $\Rplus \times \rr$-valued random variables. The random variables $W_i$ and $J_i$  can be dependent. Furthermore we define the partial sum 
%to use the same language as in the abstract!
\begin{align}\label{eq11}
S(n):=\sum_{i=1}^n W_i \hspace{0.5cm}
\end{align}
and the partial maximum
\begin{align} \label{eq11b}
M(n):=\bigvee_{i=1}^n J_i :=\max\{J_1,\dots,J_n\}. 
\end{align}
Assume now that there exist constants  $a_n,b_n>0$ and $c_n\in \mathbb{R}$, such that
\begin{eqnarray}
(a_n S(n),b_n(M(n)-c_n))\Rightarrow (D,A) \text{ as } n \longrightarrow \infty, \label{AnnahmeEinleitung}
\end{eqnarray}
where $A$ and $D$ are non-degenerate. We want to answer the following questions:
\begin{itemize}
\item[(i)] How can we characterize the joint distribution of $(D,A)$ in \eqref{AnnahmeEinleitung}? 
\item[(ii)] How can we describe the dependence between $D$ and $A$?
\item[(iii)] Are there necessary and sufficient conditions on $(W,J)$, such that the convergence in \eqref{AnnahmeEinleitung} is fulfilled?
\end{itemize}
Observe that by the classical theory of sum- or max-stability it follows by projecting on either coordinate in \eqref{AnnahmeEinleitung} that $D$ has a $\beta$ sum-stable distribution for some $0<\beta<1$ and $A$ has one of the three extreme value distributions. 
To answer all these questions we use Harmonic Analysis on the sum/max semigroup and derive a theory that subsumes both the classical theories of sum-stability or max-stability, respectively.  

This paper is organized as follows: In Section 2, by applying results from abstract harmonic analysis on semigroups to the sum/max semigroup defined by
\begin{align}\label{eq13}
(x_1,t_1)  \plusmax (x_2,t_2) := (x_1+x_2,t_1 \vee t_2)
\end{align}
for all $(x_1,t_1),(x_2,t_2) \in \Rplus \times \rr$, we develop the basic machinery. We will give a L\'evy-Khintchine type formula for sum/max infinitely divisible laws based on a modified Laplace transform on the semigroup as well as convergence criteria for triangular arrays. These methods are then used in Section 3 to answer questions (i), (ii) and (iii) in the $\alpha$-Fr\'echet case, where we additionally assume that  $A$ in \eqref{AnnahmeEinleitung} has an $\alpha$-Fr\'echet distribution. The general case 
then follows by transforming the second component in \eqref{AnnahmeEinleitung} to the $1$-Fr\'echet case.   In Section 4 we present some examples showing the usefulness of our results and methods. Technical proofs are given in the appendix.

\section{Harmonic Analysis on semigroups}

Even though the random variables $J_i$ in \eqref{eq11} are real valued, in extreme value theory it is more natural to consider random variables with values in the two-point compactification $\Rpmu=[-\infty,\infty]$. Observe that $-\infty$ is the neutral element with respect to the max operation. The framework for analyzing the convergence in \eqref{AnnahmeEinleitung} is the Abelian topological semigroup $(\Rplus\times\Rpmu, \plusmax)$, where $\plusmax$ is defined in \eqref{eq13}. Observe that the neutral element is $(0,-\infty)$. The semigroup operation $\plusmax$ naturally induces a convolution $\circledast$ on $\mathcal{M}^{b}(\Raum)$, the set of bounded measures. Indeed, let  $\Pi((s_1,y_1),(s_2,y_2)):=(s_1+s_2,y_1 \vee y_2)$.  For $\mu_1,\mu_2\in\mathcal{M}^{b}(\Raum)$ we define $\mu_1\circledast\mu_2=\Pi(\mu_1\otimes\mu_2)$, where $\mu_1\otimes\mu_2$ denotes the product measure. Then we have for independent $\Raum$-valued random vectors $(W_1,J_1)$ and $(W_2,J_2)$ that
\begin{align*}
P_{(W_1,J_1)} \circledast P_{(W_2,J_2)} = P_{(W_1,J_1) \tiny{\plusmax (W_2,J_2)}}=P_{(W_1+W_2,J_1 \vee J_2)}.
\end{align*}
The natural transformation on the space of bounded measures for the usual convolution that transforms the convolution into a product, is the Fourier- or Laplace transform. We will now introduce a similar concept on our semigroup $(\Rplus\times\Rpmu, \plusmax)$ and present its basic properties. In order to do so we first recall some basic facts about Laplace transforms on semigroups.

%Because we want to consider the sum of random variables in the first and the maximum in the second component, our measures are defined on $\Rplus \times \Rpmu$ together with the operation $\plusmax$, which is defined by 
%\begin{align}
%(x_1,t_1)  \plusmax (x_2,t_2) := (x_1+x_2,t_1 \vee t_2)
%\end{align}
%for all $(x_1,t_1),(x_2,t_2) \in \Rplus \times \Rpmu$. This is a commutative, topological semigroup with neutral element $e=(0,-\infty)$.\\
On an arbitrary semigroup $S$ a generalized Laplace transform $\mathcal{L}: \mu \rightarrow \mathcal{L}(\mu)$ is defined by 
\begin{align*}
\mathcal{L}(\mu)(s)=\int_{\bf{\hat{S}}} \rho(s) d\mu(\rho), \ s \in S,
\end{align*}
where $\bf\hat{S}$ is the set of all bounded semicharacters on $S$ and $\mu\in\mathcal M^b(\bf\hat{S})$ (see 4.2.10 in \cite{Berg}). A semicharacter on $(S,\circ)$ is a function $\rho: S \rightarrow \mathbb{R}$ with the properties
\begin{itemize}
\item[(i)] $\rho(e)=1;$
\item[(ii)] $\rho(s\circ t)=\rho(s) \rho(t) \text{ for } s,t, \in S.$
\end{itemize}
We now consider the topological semigroup $S:=(\Rplus \times\overbar{\rr},\plusmin)$ with neutral element $e=(0,\infty)$, where $\plusmin$ is defined as
\begin{align}
(x_1,t_1)  \plusmin (x_2,t_2) := (x_1+x_2,t_1 \wedge t_2)
\end{align}
for all $(x_1,t_1),(x_2,t_2) \in \Rplus \times \overbar{\rr}$. The set of bounded semicharacters on $S$ is given by
\begin{align}
\hat{\textbf{S}}&=\Big\{e^{-t\LargerCdot}\ind{\left\{\infty\right\}}(\LargerCdot),e^{-t\LargerCdot}\ind{\left[x,\infty\right]}(\LargerCdot),e^{-t\LargerCdot}\ind{\left(x,\infty\right]}(\LargerCdot)\Big| t \in \left[0,\infty\right], x \in \left[-\infty,\infty\right) \Big\}\label{MengeallerHG}
\end{align}
with $\infty \cdot s = \infty$ for $s>0$ and $0 \cdot \infty=0$, hence for $t=\infty$ we get $e^{-t \LargerCdot}=\ind{\left\{\infty\right\}}(\LargerCdot)$.  We consider only a subset of $\hat{\textbf{S}}$, which we denote by $\tilde{\textbf{S}}$:
\begin{align}
\tilde{\textbf{S}}&:=\Big\{\rho_{t,x}(s,y):=e^{-ts}\ind{[-\infty,y \,]}(x) \Big| t \in \left[0,\infty\right), x \in \left[-\infty,\infty\right]\Big\}.%=\Big\{\rho_{t,x},t \in \left[0,\infty\right], x \in \left[-\infty,\infty\right] \Big\}
\end{align}
This is again a topological semigroup under pointwise multiplication and the neutral element is the constant function $\textbf{1}$. It is easy to see that this set of semicharacters together with the pointwise multiplication is isomorph to $(\Rplus \times \Rpmu,\plusmax)$. Hence, with a little abuse of notation, by identifying measures on $\tilde{\textbf{S}}$  with measures on $\Raum$  we can define a Laplace transform for measures on $(\Rplus \times \Rpmu,\plusmax)$.

\begin{defi}
For bounded measures $\mu$ on $\Raum$, the \textit{CDF-Laplace transform} (short: \textit{C-L transform}, or \textit{C-L function}) 
$\mathcal{L}: \mu \rightarrow \mathcal{L}(\mu)$ is given by
\begin{align}
\mathcal{L}(\mu)(s,y):=\int_{\Rplus \times \Rpmu} e^{-st} \ind{\left[-\infty,y \right]}(x) \mu(dt,dx), \quad (s,y) \in \Rplus \times \Rkomp. \label{C-L-Transformierte}
\end{align} 
\end{defi}

Observe that setting $s=0$ results in the CDF of the second component, whereas setting $y=\infty$ results in the usual Laplace transform of the first component.
That is, if we consider a random vector $(W,J)$ on $\Raum$ with joint distribution $\mu$ and put $s=0$ resp. $y=\infty$ we get 
\begin{align}
&\mathcal{L}(\mu)(0,y)=\mu(\mathbb{R_{+}}\times [-\infty,y] )=P\left\{J \leq y\right\}=F_J(y) \ \text{ resp.}\\
&\mathcal{L}(\mu)(s,\infty)=\int_0^{\infty}e^{-st}P_{(W,J)}(dt,\mathbb{R_{+}})=\mathbb{E}[e^{-sW}]=\tilde{P}_W(s),
\end{align}
where $\tilde{P}_W$ is the Laplace transform of $P_W$ and $F_J$ the distribution function of $J$, which explains the name CDF-Laplace transform.\\

In the following we collect  some important properties of the C-L transform needed for our analysis.

\begin{lemma}\label{lemmaCL} 
A normalized function $\varphi$ on $(\Rplus \times \Rkomp,\plusmin)$ (meaning that $\varphi(0,\infty)=1$), is the C-L transform of a probability measure $\mu$ on $\Rplus \times \Rpmu$ if and only if $\varphi$ is positive semidefinite, $\varphi(0,y)$ is the distribution function of a probability measure on $\Rpmu$ and $\varphi(s,\infty)$ the Laplace transform of a probability measure on $\Rplus$. 
\end{lemma}
\begin{proof}See Appendix.
\end{proof}

\begin{prop} \label{EigenschaftenC-L}
Let $\mu_1, \mu_2,\mu \in \mathcal{M}^{1} (\Raum)$ and $\alpha,\beta \in \mathbb{R}$. Then
\begin{itemize}
\item[(a)] $\mathcal{L}\left(\alpha\mu_1+\beta\mu_2\right)(s,y) = \alpha\mathcal{L}(\mu_1)(s,y) +\beta\mathcal{L}(\mu_2)(s,y)$ for all $(s,y) \in \Rplus \times \Rkomp$.
\item[(b)] $\mathcal{L}(\mu_1 \circledast \mu_2)(s,y)=\mathcal{L}(\mu_1)(s,y) \cdot \mathcal{L}(\mu_2)(s,y)$ for all $(s,y) \in \Rplus \times \Rkomp$. \label{Faltungseigenschaft}
\item[(c)] $\mu_1=\mu_2$ if and only if $\mathcal{L}(\mu_1)(s,y)=\mathcal{L}(\mu_2)(s,y)$ \label{Eindeutigkeit} for all $(s,y) \in \Rplus \times \Rkomp$. 
\item[(d)] It is $0 \leq \mathcal{L}(\mu)(s,y) \leq 1$ for all $(s,y) \in \Rplus \times \Rkomp$. \label{LZwischen0und1}
\end{itemize} 
\end{prop}
\begin{proof}
Property (a) is obvious. The proof of (b) is also straightforward, because the convolution is the image measure under the mapping $T: \hat{\textbf{S}} \times \hat{\textbf{S}} \rightarrow \hat{\textbf{S}}$ with $T(\rho_1,\rho_2):=\rho_1 \rho_2$.
Property (c) follows immediately from Theorem 4.2.8 in \cite{Berg} and (d) is obvious. 
\end{proof}

The Laplace transform is a very useful tool for proving weak convergence of sums of i.i.d.\ random variables using the so called Continuity Theorem.  The next theorem is the analogue of the Continuity Theorem for the Laplace transform in the sum/max case.

\begin{theorem}[Continuity Theorem for the C-L transform] \label{Stetigkeitssatz}\ \\
Let $\mu_n,\mu \in \mathcal{M}^1(\Raum )$ for all $n \in \mathbb{N}$. Then we have:
\begin{itemize}
\item[(a)] If $\mu_n \xrightarrow[]{w}\mu$, then $\mathcal{L}(\mu_n)(s,y) \xrightarrow[]{}\mathcal{L}(\mu)(s,y)$ for all $(s,y) \in \Rplus \times \Rkomp$ in which $\mathcal{L}(\mu)$ is continuous. (This is the case for all but countably many $y \in \Rkomp$.)
\item[(b)] If $\mathcal{L}(\mu_n)(s,y) \xrightarrow[]{} \varphi(s,y)$ in all but countable many $y \in \Rkomp$ and $\lim_{s\downarrow 0}\varphi(s,\infty)=1$, then there exists a measure $\mu \in \mathcal{M}^1(\Raum)$ with $ \mathcal{L}(\mu)=\varphi$ and $\mu_n \xrightarrow[]{w} \mu$.
\end{itemize}
\begin{proof}
(a): With the Portmanteau Theorem (see for example \cite{thebook}, Theorem 1.2.2) 
we know that $\int f(t,x) \mu_n(dt,dx) \rightarrow \int f(t,x) \mu(dt,dx)$ as $n \rightarrow \infty$ for all real-valued, bounded functions $f$ on $\Raum \text{ with } \mu\left(\textnormal{Disc}(f)\right)=0$, where $\textnormal{Disc}(f)$ is the set of discontinuities of $f$. If we choose $f$ as $f_{s,y}(t,x):=e^{-st}\ind{[-\infty,y]}(x)$ it follows that 
\[\mathcal{L}({\mu}_n) (s,y) \longrightarrow \mathcal{L}(\mu)(s,y) \text{ as } n\rightarrow \infty\] 
for all $(s,y) \in \Rplus \times \Rkomp$ in which $\mathcal{L}(\mu)$ is continuous. Because $\textnormal{Disc}(f_{s,y})=\Rplus \times \left\{y\right\}$ and $\mu(\Rplus \times \cdot)$ has as probability measure at most countable many atoms, $\mu(\textnormal{Disc}(f_{s,y})) \neq 0$ for at most countable many $y\in \Rkomp$.\\
(b): Let $ (\mu_n)_{n \in \N} $ be a sequence of probability measures on $\Raum$. With Helly's Selection Theorem (see \cite{feller}, Theorem 8.6.1) we know that for all subsequences $(n_k)_{k \in \N}$ there exists another subsequence $(n_{k_l})_{l \in \mathbb{N}}$ and a measure $\mu \in \mathcal{M}^{\leq 1}(\Raum)$ such that \[\mu_{n_{k_l}} \xlongrightarrow[]{v} \mu \text{ as } l \rightarrow \infty.\]
Then $\mu$ is a subprobability measure, i.e. $\mu(\Raum)\leq 1$. With (a) it follows that \[\mathcal{L}({\mu}_{n_{k_l}})(s,y) \underset{}{\longrightarrow} \mathcal{L}({\mu})(s,y) \text{ as } l \rightarrow \infty\] for all $(s,y)$ where $\mathcal{L}({\mu})$ is continuous. By assumption we know that \[\mathcal{L}({\mu}_{n_{k_l}})(s,y) \underset{}{\longrightarrow} \varphi(s,y) \text{ as } l \rightarrow \infty\] pointwise in all but countable many $y\in \Rkomp$. Then it follows because of uniqueness of the limit that $\mathcal{L}({\mu})(s,y)=\varphi(s,y)$ for all subsequences $(n_k)_{k \in \N}$. So the limits are equal for all subsequences $(n_k)_{k \in \N}$. Because of the uniqueness of the C-L transform it follows that \[\mu_n \xlongrightarrow[]{v} \mu \text{ as } n \rightarrow \infty\] where $\mu(\Raum)\leq 1$. Because of the assumption $\lim_{s\downarrow 0}\varphi(s,\infty)=1$  we get
\begin{align*}
1&=\lim_{s \downarrow 0} \varphi(s,\infty)=\lim_{s \downarrow 0} \mathcal{L}(\mu)(s,\infty)=\lim_{s \downarrow 0} \int_{0}^{\infty} e^{-st} \mu(dt,\Rpmu)\\
&=\int_{0}^{\infty} 1 \ \mu(dt,\raum)=\mu(\Raum).
\end{align*}
Hence it is $\mu(\Raum)=1$, i.e. $\mu \in \mathcal{M}^1(\Raum)$.
\end{proof}
\end{theorem}

The following Lemma extends the convergence in Theorem \ref{Stetigkeitssatz} to a kind of uniform convergence on compact subsets needed later.  
\begin{lemma}\label{gleichmKonv} Let $\mu_n, \mu \in \mathcal{M}^1(\Raum )$ for all $n \in \mathbb{N}$ and assume that $\mathcal{L}(\mu)(s,y)$ is continuous in $y \in \Rkomp$. If $\mu_n \xlongrightarrow[]{w}\mu$ and $(s_n,y_n)\xlongrightarrow[]{}(s,y)$ then
%\begin{align*}
$\mathcal{L}(\mu_n)(s_n,y_n) \xrightarrow[]{} \mathcal{L}(\mu)(s,y)$ as $n \rightarrow \infty$.
%\end{align*}
\end{lemma}
\begin{proof}
See Appendix.
\end{proof}

As for any type of convolution structure, there is the concept of infinite divisibility. 

\begin{defi} A probability measure $\mu \in \mathcal{M}^1(\Raum)$ is \textit{infinitely divisible with respect to} $\plusmax$ (or short: $\plusmax$-\textit{infinitely divisible}), if for all $n\in \mathbb{N}$ there exists a probability measure $\mu_n \in \mathcal{M}^1(\Raum)$, such that $\mu_n^{\circledast n}=\mu.$
\end{defi}
Trivially, every distribution on $\Rpmu$ is max-infinitely divisible. The following example shows that sum-infinite divisibility in one component and max-infinite divisibility in the other component not necessarily implies $\plusmax$-infinite divisibility.
\begin{example}
Let $(X,Y)$ be a random vector which distribution is given by
\begin{itemize}
\item $P(X=k,Y=1)=\text{Pois}_{\lambda}(k) \text{ if } k \in \mathbb{N}_0 \text{ is even}$;
\item $P(X=k,Y=0)=\text{Pois}_{\lambda}(k) \text{ if } k \in \mathbb{N}_0 \text{ is odd}$;
\item $P(X=k,Y=l)=0$ for $k\in\mathbb{N}_0$, $l\geq 2$;
\end{itemize}
for a $\lambda>0$.
Furthermore the distribution of $Y$ is given by
\begin{align*}
P(Y=1)=P(Y=0)=1/2.
%\vspace{3cm}
\end{align*}
$Y$ is trivially max-infinite divisible (every univariate distribution is max-infinite divisible). The random variable $X$ is Poisson distributed with parameter $\lambda>0$ and hence sum-infinite divisible. If $(X,Y)$ is $\plusmax$-infinite divisible, there exist i.i.d. random vectors $(X_1,Y_1),(X_2,Y_2)$, such that
\begin{align*}
(X,Y)\overset{d}{=}(X_1,Y_1) \plusmax (X_2,Y_2).
\end{align*}
However, there is no distribution which fulfils this. In fact, by necessity the support of $(X_1,Y_1)$ has to be a subset of $\mathbb{N}_0 \times \left\{0,1\right\}$ and $(X_1,Y_1)$ has no mass in $(0,0)$. Consequently there exists no distribution for $(X_1,Y_1)$, such that $P(X_1+X_2=1,Y_1 \vee Y_2=0)$ is  positive. But on the other hand we have
\begin{align*}
P(X=1,Y=0)=\text{Pois}_{\lambda}(1)>0.
\end{align*} 
So $(X,Y)$ can not be $\plusmax$-infinite divisible.
\end{example} 
The next Lemma shows that the weak limit of a sequence of $\plusmax$-infinite divisible measures is $\plusmax$-infinite divisible as well.
\begin{lemma}\label{lemmawGrenzunendteil} Let $\mu_n, \mu \in \mathcal{M}^1(\Raum)$ for all $n \in \mathbb{N}$ and $\mu_n \wKonvergenz \mu$ as $n \rightarrow \infty$. If $\mu_n$ is $\plusmax$-infinite divisible for each $n \in \mathbb{N}$ then $\mu$ is $\plusmax$-infinite divisible.
\end{lemma}
\begin{proof}
See Appendix.
\end{proof}
In the following let $x_0$ denote the \textit{left endpoint} of the distribution of $A$ in \eqref{AnnahmeEinleitung}, i.e.
\begin{align*}
x_0:=\inf\left\{y \in \Rpmu: F_A(y)>0\right\}.
\end{align*}
For $F_A$ there are two possible cases, namely either $F_A(x_0)=0$ or there is an atom in $x_0$ so that $F_A(x_0)>0$. Since the limit distributions of rescaled maxima are the extreme value distributions which are continuous, in the following we will only consider the case where $F_A(x_0)=0$. If $\varphi$ is a C-L transform, we call the function $\Psi: S \rightarrow \mathbb{R}$ with
\begin{align}
\nolinebreak{\varphi=\exp(-\Psi)} \hspace{0.5cm} \text{ and } \hspace{0.5cm}  \Psi(0,\infty)=0 
\end{align}
\textit{C-L exponent} (similar to the Laplace exponent in the context of Laplace transforms). The following Theorem gives us a L\'evy-Khintchine representation for the C-L exponent of $\plusmax$-infinite divisible distributions on the semigroup $\Rplus \times \Rpmu$.
\begin{theorem}\label{TheoremLKD}\ \\
A function $\varphi$ is the C-L transform of a $\plusmax$-infinite divisible measure $\mu$ on $\Raum$ with left endpoint $x_0$ such that $\mu(\rr_+\times\{x_0\})=0$, if and only if there exists an $a \in \Rplus$ and a Radon measure $\eta$ on $\Rplus \times [x_0,\infty]$ with $\eta(\left\{(0,x_0)\right\})=0$ satisfying the integrability conditions
\begin{align}
\int_{\Rplus} \min(1,t) \eta (dt,[x_0,\infty]) < \infty \ \text{ and } \ \eta(\Rplus \times (y,\infty]) < \infty \ \forall y > x_0,  \label{levymass}
\end{align}
such that $\Psi:=-\log(\varphi)$ has the representation
\begin{align}
\Psi(s,y)=\begin{cases}\displaystyle a \cdot s +\int_{\Rplus}\int_{\left[x_0,\infty\right]}\left(1-e^{-st}\cdot\ind{\left[x_0,y\right]}(x)\right)\eta(dt,dx) & \forall y>x_0 \label{LKD}\\
\infty & \forall y \leq x_0\end{cases}
\end{align}
for all $(s,y) \in \Rplus \times \Rkomp$. The representation given in \eqref{LKD} is unique and we write $\mu \sim \left[x_0,a,\eta\right]$. We call a measure $\eta$ which fulfils \eqref{levymass} a \textit{L\'evy measure} on the semigroup $(\Rplus \times \Rpmu,\plusmax)$. 
\end{theorem}
\begin{proof}
Let $\varphi$ be the C-L transform of a $\plusmax$-infinite divisible measure $\mu$. Since  \[\varphi(s,y)=\int_{\Rplus \times \Rpmu} e^{-st} \ind{\left[-\infty,y\right]} (x) \mu(dt,dx) ,\] 
we have by our assumptions that $0<\varphi(s,y) \leq1$ for all $(s,y) \in \Rplus \times (x_0,\infty]$. On the set $\Rplus \times \left(-\infty,x_0\right]$ we have $\varphi\equiv0$ and hence $\Psi\equiv\infty$. In the following we consider $\varphi$ restricted on $S_{x_0}:=\Rplus \times (x_0,\infty]$ with the semigroup operation $\plusmin$. The function $\varphi$ is strictly positive, positive semidefinite and $\plusmax$-infinite divisible, consequently the map $\Psi : S_{x_0} \rightarrow \mathbb{R}$ with $\Psi:=-\log(\varphi)$ is due to Theorem 3.2.7 in \cite{Berg} negative semidefinite. With Theorem 4.3.20 in \cite{Berg} it then follows that there exists an additive function $q: S_{x_0} \rightarrow [0,\infty[$ and a radon measure $\tilde{\eta} \in \mathcal{M}_{+}(\bf{\hat{S}_{x_0}}\backslash \left\{1\right\})$ such that
\begin{align}
\Psi(s,y)=\Psi(e)+q(s,y)+\int_{\bf{\hat{S}_{x_0}}\backslash \left\{1\right\}}\left(1-\rho(s,y)\right)\tilde{\eta}(d\rho), \label{LKDPsi1}
\end{align}
where $\bf{\hat{S}_{x_0}}$ is the set of semicharacters on the semigroup $(S_{x_0},\plusmin)$. We now show that the additive function $q$ is of the form $q(s,y)=a \cdot s$ for some $a \geq 0$. In view of the fact that $\varphi(s,y)$ is continuous in $s$ for an arbitrary but fixed $y \in \left(x_0,\infty\right]$, $\Psi$ has to be continuous and hence also $q$ for $s>0$ (the integral in \eqref{LKDPsi1} has at most a discontinuity in $s=0$). Due to the fact that $q$ is additive we have
\begin{align*}
q(s_1+s_2,y_1 \wedge y_2)=q(s_1,y_1)+q(s_2,y_2)
\end{align*}
for any $(s_1,y_1),(s_2,y_2) \in S_{x_0}$. Because $q$ is continuous for an arbitrary but fixed $y$ in $s$ (up to $s=0$) and $q(s_1+s_2,y)=q(s_1,y)+q(s_2,y)$  there exists an $a(y) \geq 0$ such that $q(s,y)=a(y) \cdot s$. Additionally we have 
\begin{align}
q(2s,y_1 \wedge y_2)&=q(s+s,y_1 \wedge y_2)=q(s,y_1)+q(s,y_2)=a(y_1)\cdot s + a(y_2) \cdot s. \label{Beweisq1}
\end{align}  
First we assume $y_1<y_2$. Then we have
\begin{align}
q(2s,y_1 \wedge y_2)&=q(2s,y_1)=a(y_1) \cdot 2s. \label{Beweisq2}
\end{align}
If we subtract \eqref{Beweisq2} from \eqref{Beweisq1} we obtain
\begin{eqnarray*}
a(y_1) &=& a(y_2).
\end{eqnarray*}
 Due to the fact that $y_1,y_2 \in \left(x_0,\infty\right]$ were chosen arbitrarily, it follows that $a(y)$ is independent of $y$ and $q$ has the form $q(s,y)= a \cdot s$ with an $a \geq 0$. We divide the set $\bf{\hat{S}_{x_0}}$ of semicharacters in two disjoint sets 
\begin{align*}
\boldsymbol{\hat{S}^{'}_{x_0}}=\left\{e^{- t \LargerCdot}\ind{\left[x,\infty\right]}| x \in \left[x_0,\infty\right], s \in \left[0,\infty\right]\right\} ,\boldsymbol{\hat{S}^{''}_{x_0}}=\left\{e^{- t \LargerCdot}\ind{\left(x,\infty\right]}| x \in \left[x_0,\infty\right), s \in \left[0,\infty\right]\right\}.
\end{align*}
Accordingly we divide the integral in \eqref{LKDPsi1} and get due to the fact that $\bf\hat{S}^{'}_{x_0}$ and $\bf\hat{S}^{''}_{x_0}$ are isomorphic to $\left[0,\infty\right] \times \left[x_0,\infty\right]$ and $\left[0,\infty\right] \times \left[x_0,\infty\right)$, respectively,
\begin{align}
\Psi(s,y)=a \cdot s&+\int_{\left[0,\infty\right]}\int_{\left[x_0,\infty\right]}\left(1-e^{-st}\cdot\ind{\left[x,\infty\right]}(y)\right)\eta_1(dt,dx) \nonumber \\
&+\int_{\left[0,\infty\right]}\int_{\left[x_0,\infty\right)}\left(1-e^{-st}\cdot\ind{\left(x,\infty\right]}(y)\right)\eta_2(dt,dx), \label{LKDPsi2}
\end{align}
where $\eta_1$ and $\eta_2$ are radon measures on $(\Rplus \times [x_0,\infty],\plusmax)$ resp. $(\Rplus \times [x_0,\infty),\plusmax)$.
If we put $s=0$ in \eqref{LKDPsi2} we get
\begin{align*}
\Psi(0,y)&=\int_{\left[0,\infty\right]}\int_{\left[x_0,\infty\right]}(1-\ind{\left[x_0,y\right]}(x))\eta_1(dt,dx)\\
&\hspace{5mm}+\int_{\left[0,\infty\right]}\int_{\left[x_0,\infty\right)}(1-\ind{\left[x_0,y\right)}(x))\eta_2(dt,dx)\\
&=\int_{\left[x_0,\infty\right]} \ind{\left(y,\infty\right]}(x)\eta_1(\Rplus,dx)+\int_{\left[x_0,\infty\right)} \ind{\left[y,\infty\right]}(x)\eta_2(\Rplus,dx).
\end{align*}
Due to the fact that $\varphi(0,y)$ is right continuous in $y$, $\Psi(0,y)$ is right continuous in $y$, too. Consequently we have $\eta_2(\Rplus \times\left\{y\right\})=0$ for all $y>x_0$ or $\eta_2 \equiv 0$. If $\eta_2(\Rplus \times \left\{y\right\})=0$ it follows that $\eta_2(A \times\left\{y\right\})=0$ for all $A \in \mathcal{B}(\Rplus)$. Hence in both cases it follows together with \eqref{LKDPsi2} that $\Psi$ has the representation
\begin{align} \label{BeweisLKDDarst}
\Psi(s,y)=a \cdot s&+\int_{\left[0,\infty\right]}\int_{\left[x_0,\infty\right]}\left(1-e^{-st}\cdot\ind{\left[x,\infty\right]}(y)\right)\eta(dt,dx),
\end{align}
where $\eta$ is a radon measure on $\Rplus \times [x_0,\infty]$. If we put $y=\infty$ in \eqref{BeweisLKDDarst}, we get
\begin{align*}
\Psi(s,\infty)=a \cdot s+\int_{[0,\infty)}(1-e^{-st})\eta(dt,[x_0,\infty])+1_{]0,\infty[}(s)\cdot \eta(\left\{\infty\right\} \times [x_0,\infty]).
\end{align*}
Since $\Psi(s,\infty)$ is continuous in every $s\in \Rplus$ it follows that
\begin{align}
\eta(\left\{\infty\right\} \times [x_0,\infty])=0. \label{etaBeweisLKD}
\end{align}
%Consequently we get 
%\begin{align*}
%\Psi(s,y)&=a \cdot s +\int_{\left[x_0,\infty\right]}\int_{\left[0,\infty\right)}\left(1-e^{-st}\cdot\ind{\left[x,\infty\right]}(y)\right)\eta(dt,dx)\\
%&\hspace{5mm}+\eta(\left\{\infty\right\} \times \left[x_0,\infty]\right)\cdot 1_{\left]0,\infty\right[}(s)+\eta(\left\{\infty\right\} \times (y,\infty])\cdot1_{\left\{0\right\}}(s).
%\end{align*}
Consequently $\Psi$ has the representation
\begin{align}
\Psi(s,y)=a \cdot s + \int_{\left[0,\infty\right)}\int_{\left[x_0,\infty\right]} (1-e^{-st}\cdot \ind{\left[x_0,y\right]}(x))\eta(dt,dx) \  \label{LKDPsi4}
\end{align}
for all $y>x_0$ where $\eta$ is a radon measure on $\Rplus \times [x_0,\infty]$ with $\eta(\left\{(0,x_0)\right\})=0$. Since $\Psi(s,y)<\infty$ for all $(s,y) \in \Rplus \times (x_0,\infty]$ the conditions in \eqref{levymass} hold true.
\\
Conversely, assume that $\Psi$ has the representation in \eqref{LKDPsi4} for all $y > x_0$. In view of the conditions \eqref{levymass}, we get for all $(s,y) \in \Rplus \times (x_0,\infty]$ that

\begin{align*}
\Psi(s,y)&=\int_{\Rplus \times [x_0,\infty]}(1-e^{-st}\ind{[x_0,y]}(x))\eta(dt,dx)\\
&=\int_{\Rplus \times [x_0,\infty]}(1-e^{-st})\eta(dt,dx)+\int_{\Rplus \times [x_0,\infty]}e^{-st}\ind{(y,\infty]}(x)\eta(dt,dx)\\
&\leq \int_{\Rplus}(1-e^{-st}) \eta(dt,[x_0,\infty])+\eta(\Rplus \times (y,\infty])\\
&<\infty.
\end{align*}
We now define a homomorphism $h: \Rplus \times [x_0,\infty] \xrightarrow[]{} \bf{\hat{S}_{x_0}}$ by $h(t,x)=e^{-t \cdot} \ind{[x,\infty]}$ and write $\Psi$ as
\begin{align*}
\Psi(s,y)= \Psi(0,\infty)+q(s,y)+\int_{\bf{\hat{S}_{x_0}} \backslash \left\{1\right\}}(1-\rho(s,y))h(\eta)(d\rho),
\end{align*}
where $(0,\infty)$ is the neutral element on the semigroup $(S_{x_0},\plusmin)$, $q$ an additive function and $h(\eta)$ the image measure of $\eta$ under $h$.
Due to Theorem 4.3.20 in \cite{Berg} is $\Psi$ a negative definite and bounded below function on $S_{x_0}$. Hence the function $\varphi=\exp(-\Psi)$ is positive definite and due to Proposition 3.2.7 in \cite{Berg} infinite divisible. The function $\varphi(0,y)=\exp(-\Psi(0,y))$ is an uniquely determined distribution function and $\varphi(s,\infty)$ a Laplace transform due to
\begin{align*}
&\Psi(s,\infty)=a \cdot s +\int_{\Rplus}(1-e^{-st})\eta(dt,[x_0,\infty]) \text{ and }\\
&\int_{\Rplus} \min(1,t) \, \eta (dt,[x_0,\infty]) < \infty,
\end{align*}
Furthermore we have  $\Psi(0,\infty)=0$. Consequently $\varphi$ is normalized and it follows from  Lemma \ref{lemmaCL} that $\varphi$ is the  C-L transform of a measure $\mu \in \mathcal{M}^1(\Rplus \times [x_0,\infty])$. 
Since  $\varphi(s,y)=0$ for all $(s,y)\in\rr_+\times [-\infty,x_0]$ we get that  $\varphi$ is the  C-L transform of an $\plusmax$-infinite divisible probability measure $\mu$ on $\Rplus \times \Rpmu$ with $\mu(\Rplus \times [-\infty,x_0])=0$.
\end{proof}

\begin{remark}\label{LKDkorollar} If $\varphi(0,x_0)=F_A(x_0)>0$  the only difference is that the case $y=x_0$ in \eqref{LKD} has to be included in the case $y>x_0$.\\
In the following we define the L\'evy measure to be zero on $\Rplus \times [-\infty,x_0)$. Hence the C-L exponent in \eqref{LKD} can be uniquely represented by
\begin{align}
\Psi(s,y)=a \cdot s + \int_{\Rplus}\int_{\Rpmu} (1-e^{-st}\cdot \ind{\left[-\infty,y\right]}(x))\eta(dt,dx), \label{LKDkorollarFormel}
\end{align}
for all $(s,y) \in \Rplus \times (x_0,\infty]$ in the case $\varphi(0,x_0)=0$.
\end{remark}

Hereinafter we say that the set $B\subset\rr_+\times [x_0,\infty]$ is \textit{bounded away from the origin} (here we think of $(0,x_0)$ if we talk about the origin), if $\text{dist}((0,x_0),B)>0$, which means that for all $x=(x_1,x_2) \in B$ exists an $\epsilon>0$ such that $x_1>\epsilon$ or $x_2>x_0 +\epsilon$. In view of the conditions \eqref{levymass}, a L\'evy measure has the property that it assigns finite mass to all sets bounded away from the origin.  We say that  a sequence $(\eta_n)_{n \in \mathbb{N}}$ of measures 
\textit{converges  vaguely}  to a L\'evy measure $\eta$ (with left endpoint $x_0$) if
\begin{align*}
\limn \eta_n(B)=\eta(B)
\end{align*} 
for all $S \in \mathcal{B}(\Rplus \times [x_0,\infty])$ with $\eta(\partial S)=0$ and $\text{dist}((0,x_0),S)>0$. We write \[ \eta_n \xrightarrow[n\rightarrow \infty]{v^{'}} \eta.\] in this case.
\begin{remark} Let $\Psi_n,\Psi$ be C-L exponents  of $\plusmax$-infinitely divisible laws $\mu_n,\mu$, respectively, where $\mu$ has left endpoint $x_0\in[-\infty,\infty]$ with $\mu(\rr_+\times\{x_0\})=0$.
If we want to show the convergence $\Psi_n(s,y) \xrightarrow[]{} \Psi(s,y)$ as $n \rightarrow \infty$ for all $(s,y) \in \Rplus \times \Rkomp$ it is enough to show the convergence for all $(s,y) \in \Rplus \times (x_0,\infty]$. 
This is because $\mathcal{L}(\mu)(0,x_0)=0 $ and 
\[\mathcal{L}(\mu_n)(s,y) \leq \mathcal{L}(\mu_n)(0,x_0)\xrightarrow[n \rightarrow \infty]{} \mathcal{L}(\mu)(0,x_0)=0,\quad y \leq x_0,\]
meaning that \[\mathcal{L}(\mu_n)(s,y) \xrightarrow[n \rightarrow \infty]{} 0=\mathcal{L}(\mu)(s,y)\] for all $(s,y) \in \Rplus \times (-\infty,x_0]$.
\end{remark}

\begin{lemma} \label{Konvergenzsatz1}\ \\
Let $(\mu_n)_{n \in \mathbb{N}}$ be a sequence of $\plusmax$-infinite divisible probability measures on $\Raum$ with $\mu_n \sim [x_n,a_n,\eta_n]$ for each $n \in \mathbb{N}$. Then  $\mu_n\overset{w}{\longrightarrow}\mu$ where $\mu \sim [x_0,a,\eta]$ (where either $x_n\leq x_0$ for all $n \in \mathbb{N}$ or $x_n \rightarrow x_0$) if and only if
\begin{itemize}
\item[(a)] $a_{n} \xrightarrow[]{} a $ for an $a\geq 0$,
\item[(b)] ${\eta_{n}} \xrightarrow[]{v'}\eta$ \text{ and }
\item[(c)] $\underset{\epsilon\downarrow 0}{\lim}\overbar{\underset{n \rightarrow \infty}{\lim}} \int_{\left\{0 \leq t<\epsilon \right\} } t \ \eta_{n}(dt,\Rpmu)=0$.
\end{itemize}
\end{lemma}
\begin{proof}
See Appendix.
\end{proof}

\begin{lemma}\label{LemmabegleitendeVerteilung}\ \\
Let $\mu \in \WRaum$ with left endpoint $x_0$ and $c>0$. We define a probability measure $\Pi(c,\mu)$ by %with support $\Rplus \times [x_0,\infty]$ 
\begin{align*}
\Pi(c,\mu):=e^{-c}\sum_{k=0}^{\infty}\frac{c^{k}}{k!} \mu^{\circledast k}
\end{align*}
on $\Rplus \times [x_0,\infty]$, where $\mu^{\circledast 0}=\varepsilon_{(0,x_0)}$. Then $\Pi(c,\mu)$ is $\plusmax$-infinite divisible with $\Pi(c,\mu) \sim \left[x_0,0, c \cdot \mu \right]$ and $\mathcal{L}(\Pi(c,\mu))(s,y)>0$ for all $(s,y) \in \Rplus \times [x_0,\infty]$.
\end{lemma}
\begin{proof}
See Appendix.
\end{proof}

\begin{lemma}\label{LemmabegleitendeVerteilung2}\ \\
Let $\mu_n ,\nu\in \mathcal{M}^{1}(\Raum)$ for each $n \in \mathbb{N}$ with left endpoints $x_n$ and $x_0$, resp., where either $x_n \rightarrow x_0$ or $x_n\leq x_0$ for each $n\in \mathbb{N}$. Then the following are equivalent:
\begin{itemize}
\item[(i)] $\displaystyle \Pi(n,\mu_{n}) \wKonvergenz \nu \text{ as } n \rightarrow \infty$; 
\item[(ii)] $\displaystyle \mu_{n}^{\circledast n} \wKonvergenz \nu \text{ as } n \rightarrow \infty$.
\end{itemize}
\end{lemma}
\begin{proof}
See Appendix.
\end{proof}

Finally, the following theorem gives convergence criteria for triangular arrays on  $(\Raum,\plusmax)$.
\begin{theorem}\label{Konvergenzsatz2}\ \\
Let $\mu_n  \in  \mathcal{M}^1(\Raum)$ for each $n \in \mathbb{N}$ with left endpoint $x_n$. Then $\mu_n^{\circledast n }\wKonvergenz \nu$ as $n \rightarrow \infty$, where $\nu$ $\plusmax$-infinite divisible, $\nu \sim \left[x_0,0,\Phi \right]$ (where either $x_n \longrightarrow x_0$ or $x_n\leq x_0$ for all $n \in \mathbb{N}$) if and only if
\begin{itemize}
\item[(a)] $n \cdot \mu_{n} \xrightarrow[]{v^{'}} \Phi$ and
\item[(b)] $\underset{\epsilon \downarrow 0}{\lim}\underset{n \rightarrow \infty}{\overline{\lim}} n \cdot \int_{\left\{0 \leq t<\epsilon\right\}}t \ \mu_{n}(dt,\Rpmu)=0$.
\end{itemize}
\begin{proof}
With Lemma \ref{LemmabegleitendeVerteilung2} and Theorem \ref{Konvergenzsatz1} this assertion now follows easily. In view of Lemma \ref{LemmabegleitendeVerteilung2} $\mu_n^{\circledast n }\wKonvergenz \nu$ is equivalent to $\Pi(n,\mu_{n}) \wKonvergenz \nu$. With Lemma \ref{LemmabegleitendeVerteilung} we know that $\Pi(n,\mu_{n}) \sim \left[x_n,0, n \cdot \mu_{n} \right]$. Hence we get with Lemma \ref{Konvergenzsatz1} that $\Pi(n,\mu_{n}) \wKonvergenz \nu$ is equivalent to
\begin{align*}
n \cdot \mu_{n} \vageKonvergenz \Phi  \ \text{  and  } \ \underset{\epsilon \downarrow 0}{\lim}\underset{n \rightarrow \infty}{\overline{\lim}} n \cdot \int_{\left\{0 \leq t<\epsilon\right\}}t \ \mu_{n}(dt,\Rpmu)=0,
\end{align*}
where $\nu \sim \left[x_0,0,\Phi\right]$.
\end{proof}
\end{theorem} 

\section{Joint convergence}

This section contains the main results of this paper. Using the methods developed in Section 2 above, we answer questions (i), (ii) and (iii) from the introduction. This will be done by
first considering the case that $A$ in \eqref{AnnahmeEinleitung} has an $\alpha$-Fr\'echet distribution for some $\alpha>0$. The general case will then be dealt with, by transforming the second component in \eqref{AnnahmeEinleitung} to the 1-Fr\'echet case, a standard technique in multivariate extreme value theory (see e.g. \cite{resnick}, p. 265).  

Our first result partially answers question (i). As expected, the non-degenerate limit distributions in \eqref{AnnahmeEinleitung} are sum-max stable in the sense of the following definition.

\begin{defi}\label{defsummaxstab}
Let $(D,A)$ be a $\rr_+\times\rr$-valued random vector with non-degenerate marginals. We say that $(D,A)$ is {\it sum-max stable}, if for all $n\geq 1$ there exist numbers $a_n,b_n>0$ and $c_n\in\rr$ such  that for i.i.d. copies $(D_1,A_1),\dots,(D_n,A_n)$ of $(D,A)$ we have
\[(D_1,A_1)\plusmax\cdots\plusmax(D_n,A_n)=\bigl(D_1+\cdots+D_n,A_1\vee\dots\vee A_n)\eqd(a_n^{-1}D,b_n^{-1}A+c_n) .\]
\end{defi}

\begin{theorem}
Let $(D,A)$ be $\rr_+\times\rr$-valued with non-degenerate marginals. Then $(D,A)$ is sum-max stable if and only if $(D,A)$ is a limit distribution in \eqref{AnnahmeEinleitung}.
\end{theorem}

\begin{proof}
Trivially, every sum-max stable random vector is a limit distribution in \eqref{AnnahmeEinleitung}. Now assume that $(D,A)$ is a non-degenerate limit distribution in \eqref{AnnahmeEinleitung}. Fix any $k\geq 2$. Then we have
\begin{equation}\label{stern2}
\bigl(a_{nk}S(nk),b_{nk}(M(nk)-c_{nk})\bigr)\Longrightarrow (D,A)\quad\text{as $n\to\infty$.}
\end{equation}
For $i=1,\dots,k$ let
\[ (S_n^{(i)},M_n^{(i)})=\bigl(\sum_{j=1}^nW_{n(i-1)+j},\bigvee_{j=1}^n J_{n(i-1)+j}\bigr) \]
so that $(S_n^{(1)},M_n^{(1)}),\dots,(S_n^{(k)},M_n^{(k)})$ are i.i.d. Moreover, by \eqref{AnnahmeEinleitung}
\[ (a_nS_n^{(i)},b_n(M_n^{(i)}-c_n))\Longrightarrow (D_i,A_i)\quad\text{as $n\to\infty$,}\]
where $(D_1,A_1),\dots,(D_k,A_k)$ are i.i.d. copies of $(D,A)$. Then we have
\begin{multline*}
\bigl(a_nS_n^{(1)},b_n(M_n^{(1)}-c_n)\bigr)\plusmax\cdots\plusmax\bigl(a_nS_n^{(k)},b_n(M_n^{(k)}-c_n)\bigr) \\ =\bigl(a_nS(nk),b_n(M(nk)-c_n)\bigr)  \Longrightarrow (D_1,A_1)\plusmax\cdots\plusmax (D_k,A_k)\quad\text{as $n\to\infty$.}
\end{multline*}
Hence, in view of \eqref{stern2}, convergence of types yields 
\[ \frac{a_{nk}}{a_n}\to \tilde a_k>0,\quad \frac{b_{nk}}{b_n}\to \tilde b_k>0\quad\text{and}\quad b_{nk}c_n-c_nb_{nk}\to \tilde c_k \]
as $n\to\infty$ and therefore
\[ (D_1,A_1)\plusmax\cdots\plusmax (D_k,A_k)\eqd (\tilde a_k^{-1}D,\tilde b_k^{-1}A+\tilde c_k) \]
so $(D,A)$ is sum-max stable.
\end{proof}

\begin{defi}\label{Definitionstabil1} Let $(D,A)$ be a $\Rplus \times \mathbb{R}$-valued random vector.
We say that the random vector $(W,J)$  belongs to the \textit{sum-max domain of attraction} of $(D,A)$,  if \eqref{AnnahmeEinleitung} holds for i.i.d.\ copies $(W_i,J_i)$  of $(W,J)$. We write $(W,J) \in \text{sum-max-DOA}(D,A)$. If $c_n=0$ in \eqref{AnnahmeEinleitung}, we say $(W,J)$ belongs to the strict sum-max-DOA of $(D,A)$ and write $(W,J) \in \text{sum-max-DOA}_{\textbf{S}}(D,A)$.
\end{defi}

\begin{cor}
Let $(D,A)$ be $\rr_+\times\rr$-valued with non-degenerate marginals. Then $(D,A)$ is sum-max stable if and only if $\text{sum-max-DOA}(D,A)\neq\emptyset$.
\end{cor}

The next theorem characterizes the sum-max domain of attraction of $(D,A)$ in the case where $A$ has an $\alpha$-Fr\'echet distribution.

\begin{theorem}\label{Theoremjointsum/max} Let $(W,J),(W_i,J_i)_{i \in \mathbb{N}}$ be i.i.d. $\Rplus \times \mathbb{R}$-valued random vectors. Furthermore assume that $(D,A)$ is a $\Rplus \times \mathbb{R}$-valued random vector, where $D$ is strictly $\beta$-stable with $0<\beta<1$ and $A$ is $\alpha$-Fr\'echet distributed with $\alpha > 0$. Then the following are equivalent:
\begin{itemize}
\item[(a)] $(W,J) \in \textnormal{sum-max-DOA}_{S}(D,A)$.
%, i.e. there exist sequences $(a_n)_{n\in \mathbb{N}}$,$(b_n)_{n \in \mathbb{N}}$ with $a_n,b_n>0$ such that \[(a_nS(n),b_nM(n)) \xLongrightarrow[n\rightarrow \infty]{} (D,A).\]
\item[(b)] There exist sequences $(a_n)_{n\in \mathbb{N}}$,$(b_n)_{n \in \mathbb{N}}$ with $a_n,b_n>0$ such that \[n \cdot P_{(a_nW,b_nJ)} \vKonvergenzninfty \eta,\] where $\eta$ is a L\'evy measure on $(\Raum,\plusmax)$.   \label{gemvageKonvergenz}
\end{itemize}

Then $(D,A)$ is sum-max stable and has the L\'evy representation $[0,0,\eta]$. We can use the same sequences $(a_n)_{n \in \mathbb{N}}$ and $(b_n)_{n \in \mathbb{N}}$ in (a) and (b). Furthermore $(a_n)$ is regularly varying with index $-1/\beta$ and $(b_n)$ is regularly varying with index $-1/\alpha$.
\end{theorem}

\begin{remark}
Since the left endpoint of the Fr\'echet distribution is $x_0=0$, the convergence in (b) means $n \cdot P_{(a_nW,b_nJ)}(B) \overset{v'}{\longrightarrow} \eta(B)$ as $n \rightarrow \infty$ for all $B \in \mathcal{B}(\Rplus^2)$ with $\eta(\partial B)=0$ and $\textnormal{dist}((0,0),B)>0$.
\end{remark}

\begin{proof}
That assertion (a) implies (b) follows directly with Theorem \ref{Konvergenzsatz2}. We assume that for sequences $(a_n)_{n \in \mathbb{N}}, (b_n)_{n \in \mathbb{N}}$ with $a_n>0$ and $b_n>0$ we have
\begin{align}\label{FrechetDOA1}
(a_nS(n),b_nM(n)) \WKonvergenz (D,A).
\end{align}
We denote \[\mu_n:=P_{(a_nW,b_nJ)} \hspace{0.5cm} \text{ and } \hspace{0.5cm} \mu:=P_{(W,J)}.\] 
Since $(W_i,J_i)_{i \in \mathbb{N}}$ are i.i.d. and distributed as $(W,J)$ equation \eqref{FrechetDOA1} is equivalent to
\begin{align*}
\mu_n^{\circledast n} \wKonvergenzninfty P_{(D,A)}, \text{ where } P_{(D,A)} \sim \left[0,0,\eta \right].
\end{align*}
Let $F(x)=P\{J\leq x\}$ denote the distribution function of $J$.  In case that the left endpoint of $F$ is $-\infty$, the left endpoint of $F(b_n^{-1}x)$ is equal to $-\infty$ for each $n$. If the left endpoint of $F$ is any real number, the left endpoint of $F(b_n^{-1}x)$ converges as $n \rightarrow \infty$ to $x_0=0$. With Theorem \ref{Konvergenzsatz2} it then follows that \[n \cdot P_{(a_nW,b_nJ)} \vKonvergenzninfty \eta.\]
That (b) implies (a) follows with Theorem \ref{Konvergenzsatz2} as well, if we show that 
\begin{align}\label{FrechetDOA2}
n \cdot P_{(a_n W, b_n J)} \vKonvergenzninfty \eta 
\end{align}
implies that
\begin{align*}
 \ \underset{\epsilon \downarrow 0}{\lim}\underset{n \rightarrow \infty}{\overline{\lim}} n \cdot \int_{\left\{0\leq t<\epsilon\right\}} t \ \mu_n(dt,\mathbb{R})=0.
\end{align*}
Due to $W\in \text{DOA}_{S}(D)$ this follows as in the proof for the domain of attraction theorem for stable distributions (see Theorem 8.2.10 in \cite{thebook}). That the sequences $(a_n)_{n \in \mathbb{N}}$ and $(b_n)_{n \in \mathbb{N}}$ are regularly varying follows from the classical theories by projecting on either component.
\end{proof}

%\begin{cor}\ \\
%Under the assumptions of Theorem \ref{Theoremjointsum/max}, where we additionally assume that the measure $\eta$ is $\mathcal{M}$-full, the following is equivalent
%\begin{itemize}
%\item[(i)] $(W,J) \in \textnormal{sum-max-DOA}_{S}(D,A)$;
%\item[(ii)] $P_{(W,J^{+})} \in RVM_{\infty}(E) \text{ with } E=\textnormal{diag}(1/\beta,1/\alpha)$.
%\end{itemize}
%\end{cor}
The measure $\eta$ in (b) in Theorem \ref{gemvageKonvergenz} has a scaling property as shown next.
\begin{cor} \label{levymaprop} \ \\
For the L\'evy measure $\eta$ in (b) of Theorem \ref{gemvageKonvergenz} we have for all $B \in \mathcal{B}(\Rplus^2)$ that
\begin{align}
(t \cdot \eta) (B) = (t^{E} \eta)(B) \text{ for all } t>0 \label{skalproperty}
\end{align}
with $E=\textnormal{diag}(1/\beta,1/\alpha)$, where $t^E=\diag(t^{1/\beta},t^{1/\alpha})$.
\end{cor}

\begin{proof}
Since $(a_n)_{n \in \mathbb{N}} \in \textnormal{RV}(-1/\beta)$ and $(b_n)_{n \in \mathbb{N}} \in \textnormal{RV}(-1/\alpha)$ in Theorem \ref{gemvageKonvergenz} we know that $\diag(a_n,b_n) \in \textnormal{RV}(-E)$ in the sense of Definition 4.2.8 of \cite{thebook}. Observe that
\begin{align*}
\mathcal{L}(P_{\big(a_n\sum_{i=1}^{\lfloor nt \rfloor} W_i, b_n \bigvee_{i=1}^{\lfloor nt \rfloor} J_i\big)})(\xi,x)&=(\mathcal{L}(P_{(a_nW_i,b_nJ_i)})^{\frac{\lfloor nt \rfloor}{n}})^{n}(\xi,x)\\
&\hspace{5mm}\xrightarrow[n \rightarrow \infty]{}\mathcal{L}(P_{(D,A)})^{t}(\xi,x),
\end{align*}
so that
\begin{align*}
P_{(a_n\sum_{i=1}^{\lfloor nt \rfloor} W_i, b_n \bigvee_{i=1}^{\lfloor nt \rfloor} J_i)} \xrightarrow[n \rightarrow \infty]{w} P_{(D,A)}^{t} \sim [0,0,t \cdot \eta],
\end{align*}
where $P_{(D,A)}^{t}$ is for $t>0$ defined as the distribution which C-L transform is given by $\mathcal{L}(P_{(D,A)})^{t}(\xi,x)$ and hence has the L\'evy representation $[0,0,t \cdot \eta]$.\\ On the other hand we get using   $a_na_{\lfloor nt \rfloor}^{-1}\rightarrow t^{1/\beta}$ and $b_nb_{\lfloor nt \rfloor}^{-1} \rightarrow t^{1/\alpha}$ as $n \rightarrow \infty$ that
\begin{align*}
P_{(a_n\sum_{i=1}^{\lfloor nt \rfloor} W_i, b_n \bigvee_{i=1}^{\lfloor nt \rfloor} J_i)}&=P_{\big(a_na_{\lfloor nt \rfloor}^{-1}a_{\lfloor nt \rfloor}\sum_{i=1}^{\lfloor nt \rfloor} W_i,b_nb_{\lfloor nt \rfloor}^{-1}b_{\lfloor nt \rfloor}\bigvee_{i=1}^{\lfloor nt \rfloor} J_i\big)}\\
&\hspace{5mm}\xrightarrow[n \rightarrow \infty]{w} P_{t^{E}(D,A)} \sim [0,0,t^{E}\eta].
\end{align*}
Because of the uniqueness of the L\'evy-Khintchine representation the assertion follows.
\end{proof}
One of our aims was to describe possible limit distributions that can appear as limits of the sum and the maximum of i.i.d.\ random variables. We call these limit distributions sum-max stable. Due to the harmonic analysis tools in Section 2 we have a method to describe sum-max infinite divisible distributions, namely by the L\'evy-Khintchine representation (see Theorem \ref{TheoremLKD}). The sum-max stable distributions are a special case of sum-max infinite divisible distributions and the next theorem describes the sum-max stable distributions by a representation of its L\'evy measure.

\begin{theorem}(Representation of the L\'evy measure) \label{TheoremLevyMaDarst} \\
Under the assumptions of Theorem \ref{gemvageKonvergenz}, there exist constants  $C\geq 0, K>0$ and a probability measure $\omega \in \mathcal{M}^1(\mathbb{R})$ with $\omega(\Rplus)>0$ and $\int_{0}^{\infty} x^{\alpha} \omega(dx)<\infty$ such that the L\'evy measure $\eta$ of $P_{(D,A)}$ on $\Rplushochzwei$ is given by
\begin{align}
\eta(dt,dx)=\epsilon_{0}(dt) C \alpha x^{-\alpha-1} dx + \ind{(0,\infty) \times \raum} (t,x)\big(t^{\beta / \alpha}\omega\big)(dx) K \beta t^{-\beta-1} dt. \label{LevyMaDarst2}
\end{align}
\end{theorem}

\begin{proof}
First we define two measures
\begin{align*}
\eta_1\left((r,\infty) \times B_1 \right):=\eta\left((r,\infty) \times B_1\right) \text{ and } \eta_2(B_2):=\eta(\left\{0\right\} \times B_2)
\end{align*}
for all Borel sets $B_1 \in \mathcal{B}(\mathbb{\raum})$, $B_2 \in \mathcal{B}((0,\infty))$ and $r>0$.
The L\'evy measure $\eta$ on $\Rplushochzwei\backslash\left\{(0,0)\right\}$ of the limit distribution $P_{(D,A)}$ can then be represented by
\begin{align*}
\eta(dt,dx)=\epsilon_{0}(dt)\eta_2(dx)+\ind{(0,\infty)\times \raum}(t,x) \eta_1(dt,dx).
\end{align*}
With Corollary \ref{levymaprop} we get for all $t>0$ setting $E=\textnormal{diag}(1/\beta,1/\alpha)$
\begin{align}
t \cdot \eta_2(B_2)=(t^{E}\eta)(\left\{0\right\} \times B_2) =\eta(\left\{0\right\} \times t^{-1/\alpha} B_2)=(t^{1/\alpha}\eta_2)(B_2). \label{BeweisLevyMaDarst1}
\end{align}
The measure $\eta_2$ is a L\'evy measure of a probability distribution on the semigroup $(\Rplus,\vee)$. If $\eta_2\not \equiv 0$, there exists a distribution function F on $\Rplus$, such that $F(y)=\exp(-\eta_2(y,\infty))$ $\text{ for all } y>0$ . From \eqref{BeweisLevyMaDarst1} it follows that
\begin{align*}
F(y)^t=F(t^{-1/\alpha}y) \text{ for all } t>0 \text{ and } y>0.
\end{align*}
Hence it follows (see proof of Proposition 0.3. in \cite{resnick}) that
$F(y)=\exp(-C y^{-\alpha})$ with $C>0$ for all $y>0$. So the measure $\eta_2$ on $\mathcal{B}(\left(0,\infty\right))$ is given by
\begin{align}\label{Darstellungvoneta2}
\eta_2(dx)=C \alpha x^{-\alpha-1} dx.
\end{align}
The measure $\eta_2$ can also be the zero measure and so $\eta_2$ has the representation \eqref{Darstellungvoneta2} with $C\geq 0$. We still have to show that $\eta_1$ has the representation
\begin{align*}
\eta_1(dt,dx)=(t^{\beta / \alpha}\omega)(dx)K\beta t^{-\beta -1} dt.
\end{align*} 
For $B_1 \in \mathcal{B}(\Rplus)$ and $r>0$ we define the set
\begin{align*}
T(r,B_1):=\left\{(t,t^{\beta / \alpha} x) : t>r, x \in B_1 \right\}.
\end{align*} 
All sets of this form are a $\cap$-stable generator of $\mathcal{B}( (0,\infty) \times \raum)$. This follows because the map $(t,x) \rightarrow (t,t^{\beta / \alpha} x)$ is a homeomorphism from $(0,\infty) \times \raum$ onto itself. Furthermore we have $T(r,B_1)=r^{\beta E} T(1,B_1)$ with $E=\textnormal{diag}(1/\beta,1/\alpha)$ and so we get with equation \eqref{skalproperty} that
\begin{align}
\eta_1(T(r,B_1))=\eta_1(r^{\beta E} T(1,B_1)=(r^{-\beta E} \eta_1)(T(1,B_1))=r^{-\beta}\cdot \eta_1(T(1,B_1)). \label{BeweisLevyMaDarst2}
\end{align}
Additionally we get for any probability measure $\omega$ on $\rr$ and a constant $K>0$
\begin{align}
\int_{T(r,B_1)} (t^{\beta / \alpha} \omega) (dy) K \beta t^{-\beta -1} dt &=\int_{r}^{\infty}\int_{t^{\beta / \alpha}B_1} (t^{\beta/\alpha}\omega)(dy)K\beta t^{-\beta -1}dt \nonumber \\
&=\int_{r}^{\infty} \omega(B_1) K \beta t^{-\beta -1} dt=\omega(B_1)Kr^{-\beta}. \label{BeweisLevyMaDarst3}
\end{align}
We define $\omega(B_1):= \frac{1}{K} \eta_1(T(1,B_1))$ where $K$ is given by $K:=\eta_1(T(1,\raum))>0$, since $\eta_1\not \equiv 0$, because of non-degeneracy and the fact that $T(1,\raum)$ is bounded away from zero. It then follows with \eqref{BeweisLevyMaDarst2} and \eqref{BeweisLevyMaDarst3} that
\begin{align*}
\eta_1(T(r,B_1))=r^{-\beta} \eta_1(T(1,B_1))=\omega(B_1) r^{-\beta}K=\int_{T(r,B_1)}(t^{\beta/\alpha}\omega)(dy) K \beta t^{-\beta -1} dt
\end{align*}
for all $r>0$ and $B_1 \in \mathcal{B}(\raum)$. Altogether it follows that the L\'evy measure has the representation \eqref{LevyMaDarst2}. Since $\eta_1$ is a L\'evy measure on $(\Raum,\plusmax)$, it necessarily satisfies condition \eqref{levymass} so that for all $y>0$ we have $\eta_1(\rr_+\times(y,\infty)<\infty$. Using the above established representation of $\eta_1$, a simple calculation shows that this is equivalent to $\int_0^\infty x^\alpha\,\omega(dx)<\infty$. This concludes the proof.
\end{proof}

With the next theorem we are able to construct random vectors which are in the sum-max domain of attraction of particular sum-max stable distributions. A random variable $W$ is in the strict domain of normal attraction of a $\beta$-stable random variable $D$ (short: $W \in \text{DONA}_S(D)$) if one can choose the normalizing constant $a_n=n^{-1/\beta}$. That means we have
\begin{align*}
n^{-1/\beta} S(n) \underset{n \rightarrow \infty}{\Longrightarrow} D. 
\end{align*}
\begin{theorem}\label{RückLevyMaDarst}\ \\
Let $(W_i)_{i \in \mathbb{N}}$ be a sequence of i.i.d. $\Rplus$-valued random variables with $W\overset{d}{=}W_i$ and $W\in \textnormal{DONA}_S(D)$, where $D$ is strictly $\beta$-stable with $0<\beta<1$ and $E\left[e^{-sD}\right]=\exp\left(-K \Gamma(1-\beta)s^{\beta}\right)$ with $ K>0$ for $s \geq 0$.  \eqref{AnnahmeEinleitung}.
Further $(\bar{J_i})_{i \in \mathbb{N}}$ are i.i.d. $\mathbb{R}$-valued random variables with 
\begin{align}
P\left(\bar{J_i} \in B_2 | W_i=t\right)=(t^{\beta / \alpha} \omega) (B_2) \ \forall B_2 \in \mathcal{B}(\mathbb{R}),
\end{align}
where $\omega$ is a probability measure on $\mathbb{R}$ with $\omega(\Rplus)>0$ and $\int_{0}^{\infty} x^{\alpha} \omega(dx)<\infty$.
Then the sequence $(W_i,\bar{J_i})_{i \in \mathbb{N}}$ fulfils \eqref{FrechetDOA1}
with $a_n=n^{-1/\beta}$, $b_n=n^{-1 / \alpha}$ and a limit distribution $P_{(D,A)}$ which L\'evy measure $\eta$ has the form \eqref{LevyMaDarst2} with $C=0$.\\
Furthermore, if we choose i.i.d. $(\tilde{J_i})_{i \in \mathbb{N}}$ with $P(\tilde{J_i} \leq x)=\exp(-Cx^{-\alpha})$ with $C>0$ for all $x>0$ and such that $(W_i,\bar{J_i})$ and $\tilde{J_i}$ are independent for all $i \in \mathbb{N}$, and we define $J_i$ by $J_i:=\tilde{J_i} \vee \bar{J_i}$, then $(W_i,J_i)_{i \in \mathbb{N}}$ fulfil \eqref{FrechetDOA1} with $a_n=n^{-1 / \beta}, b_n=n^{-1 / \alpha}$ and a limit distribution $P_{(D,A)}$ which L\'evy measure has the representation \eqref{LevyMaDarst2} with $C>0$.
\end{theorem}

\begin{proof}
We first consider the case $C=0$. In view of Theorem \ref{Theoremjointsum/max} it is enough to show that for any continuity set $B\in \mathcal{B}(\Rplushochzwei)$ with $\text{dist}((0,0),B)>0$ we have 
\begin{align}\label{kkk}
n \cdot P_{(n^{-1/\beta} W, n^{-1/\alpha} \bar{J})}(B) \xrightarrow[n \rightarrow \infty]{} \eta_1(B),
\end{align}
where $\eta_1$ is given by \eqref{LevyMaDarst2} with $C=0$.
First let $r>0$ and $x\geq 0$. Then we get
\begin{align*}
n  P_{(n^{-1/\beta} W, n^{-1/\alpha} \bar{J})}((r,\infty) \times (x,\infty)) &=n\cdot P(W>n^{1/\beta}r,\bar{J} >n^{1/\alpha}x)\\
&=n \int_{0}^{\infty} P(\bar{J} > n^{1/\alpha} x | W=t) \ind{(r,\infty)}(n^{-1/\beta} t) P_W(dt) \\
&= n  \int_{0}^{\infty} (t^{\beta/\alpha} \omega ) ( n^{1/\alpha}x,\infty)  \ind{(r,\infty)}(n^{-1/\beta} t) P_W(dt) \\
&= n  \int_{r}^{\infty} (t^{\beta/\alpha} \omega ) ( x,\infty ) P_{n^{-1/\beta}W}(dt)  \\
&\Konvergenzninfty \int_{r}^{\infty} (t^{\beta/\alpha} \omega ) (x,\infty) ) K \beta t^{-\beta-1} dt \\ & =\eta_1((r,\infty) \times (x,\infty)),
\end{align*}
where the last step follows from Proposition 1.2.20 in \cite{thebook}, since the set $(r,\infty)$ is bounded away from zero and furthermore the map $t \rightarrow (t^{\beta/\alpha}\omega)(x,\infty)$ is continuous and bounded.
On the other hand, for $r\geq 0$ and $x>0$ we get
\begin{align*}
n  P_{(n^{-1/\beta} W, n^{-1/\alpha} \bar{J})}(  (r,\infty)\times (x,\infty) )&= n P(W>n^{1/\beta}r,\bar{J} > n^{1/\alpha}x)\\
&= n \int_r^\infty (t^{\beta/\alpha} \omega ) (x,\infty) P_{n^{-1/\beta}W}(dt) \\
&= n\int_0^\infty P(n^{-1/\beta}W>\max(r, (u/x)^{-\alpha/\beta}))\,\omega(du) .
\end{align*}
Observe that
\begin{multline*}
 nP(n^{-1/\beta}W>\max(r,(u/x)^{-\alpha/\beta}))\\  \to \int_{\max(r,(u/x)^{-\alpha/\beta})}^\infty K\beta t^{-\beta-1}\,dt=K\max(r,(u/x)^{-\alpha/\beta})^{-\beta},
 \end{multline*}
 as $n\to\infty$. Moreover, since $W\in \textnormal{DONA}_S(D)$ we know that there exists a constant $M>0$ such that $P(W>t)\leq Mt^{-\beta}$ for all $t>0$. Hence
 \[ nP(n^{-1/\beta}W>\max(r,(u/x)^{-\alpha/\beta})\leq nP(W>n^{1/\beta}(u/x)^{-\alpha/\beta})\leq Mx^{-\alpha}u^\alpha .\]
Since by assumption $\int_0^\infty u^\alpha\,\omega(du)<\infty$, dominated convergence yields
\begin{multline*}
n  P_{(n^{-1/\beta} W, n^{-1/\alpha} \bar{J})}(  (r,\infty)\times (x,\infty) ) \\  \to \int_0^\infty K \max(r,(u/x)^{-\alpha/\beta})^{-\beta}\omega(du)=\eta_1((r,\infty)\times(x,\infty))
\end{multline*}
as $n\to\infty$ again. Hence we have shown, that for $r,x\geq 0$ with $\max(x,r)>0$ we have 
\[ n  P_{(n^{-1/\beta} W, n^{-1/\alpha} \bar{J})}(  (r,\infty)\times (x,\infty) )\to \eta_1((r,\infty)\times(x,\infty))
\quad\text{as $n\to\infty$,} \]
which implies \eqref{kkk}.
In view of Theorem \ref{Theoremjointsum/max} we therefore have 
\begin{align*}
(n^{-1/\beta} \sum_{i=1}^{n} W_i, n^{-1/\alpha}\bigvee_{i=1}^{n} \bar{J_i}) \WKonvergenz (D,\bar{A})
\end{align*}
and the L\'evy measure $\eta_{1}$ of $(D,\bar{A})$ is given by \eqref{LevyMaDarst2} with $C=0$.\\ If we now choose  a sequence of i.i.d. and $\alpha$-Fr\'echet distributed random variables $(\tilde{J_i})_{i \in \mathbb{N}}$ with $P(\tilde{J_i}\leq x):=\exp(-Cx^{-\alpha})$ which are independent of $(W_i,\bar{J_i})$ it follows
\begin{align*}
[(0, n^{-1/\alpha} \bigvee_{i=1}^{n} \tilde{J}_i),(n^{-1/\beta} \sum_{i=1}^{n} W_i,n^{-1/\alpha}\bigvee_{i=1}^{n} \bar{J}_i)] \WKonvergenz [(0,\tilde{A}),(D,\bar{A})].
\end{align*}
The distribution of $(0,\tilde{A})$ has the L\'evy measure $\eta_1(dt,dx)=\epsilon_{0}(dt)C\alpha x^{-\alpha-1}$. Since $(W_i,\bar{J_i})$ and $\tilde{J_i}$ are independent, the random vectors $(0,\tilde{A})$ and $(D,\bar{A})$ are also independent. With the continuous mapping theorem applied to the semigroup operation $\plusmax$ it then follows that
\begin{align*}
(n^{-1/\beta} \sum_{i=1}^{n} W_i, n^{-1/\alpha}\bigvee_{i=1}^{n} J_i) \WKonvergenz \left(D,A\right)
\end{align*}
where $A:=\tilde{A} \vee \bar{A}$. Hence the the L\'evy measure of the distribution of $(D,A)$ is $\eta:=\eta_1+\eta_2$ and thus has the representation in \eqref{LevyMaDarst2} with $C>0$.
\end{proof}
The next Corollary characterizes the case of asymptotic independence i.e.\ $D$ and $A$ are independent.
\begin{cor}\label{UnabhFrechetFall}\ \\
The random variables $A$ and $D$ in Theorem \ref{Theoremjointsum/max} are independent if and only if in the L\'evy representation in \eqref{LevyMaDarst2} we have $C>0$ and $\omega=\epsilon_0$.
\end{cor}
\begin{proof}
If $A$ and $D$ are independent, the L\'evy measure has the representation
\begin{align*}
\eta(dt,dx)= \epsilon_{0} (dt) \Phi_{A}(dx)+\epsilon_0(dx) \Phi_D(dt),
\end{align*}
where $\Phi_{A}(dx)=C \alpha x^{-\alpha-1} dx$ with $C>0$, $\alpha>0$ and $\Phi_{D}(dt)=K\beta t^{-\beta-1} dt$ with $K>0$,$0<\beta<1$. With Theorem \ref{TheoremLevyMaDarst} the L\'evy measure has the representation \eqref{LevyMaDarst2}. The uniqueness of the L\'evy measure implies that $C>0$ and $t^{\beta/\alpha}\omega=\epsilon_0$, hence we get $\omega=\epsilon_0$. Conversely, if $C>0$ and $\omega=\epsilon_0$, the L\'evy measure is given by
\begin{align*}
\eta(dt,dx)=\epsilon_{0}(dt) C \alpha x^{-\alpha-1} dx + \epsilon_0(dx) K \beta t^{-\beta -1} dt.
\end{align*}
This implies that the C-L exponent of $(D,A)$ is
\begin{align*}
\Psi(s,y)&=\int_{\Rplus^2} (1-e^{-st} \ind{[0,y]}(x)) \epsilon_{0}(dt) C\alpha x^{-\alpha-1} dx\\ &\hspace{5mm}+\int_{\Rplus^2}(1-e^{-st} \ind{[0,y]}(x)) \epsilon_{0}(dx)K\beta t^{-\beta-1}dt\\
&=-\log F_A(y)+\Psi_D(s),
\end{align*} 
which implies that $A$ and $D$ are independent.
\end{proof}

The following Proposition delivers us a representation for the C-L exponent of the sum-max stable distributions in the $\alpha$-Fr\'echet case.
\begin{prop}\label{LaplaceExponent}\ \\
The C-L exponent of the limit distribution $P_{(D,A)} \sim \left[0,0,\eta\right]$ in Theorem \ref{TheoremLevyMaDarst} is given by
\begin{align}
\Psi(s,y)=K\Gamma(1-\beta)s^{\beta}+y^{-\alpha} \left(C+ \intnu e^{-sty^{\alpha / \beta}} \omega(t^{-\beta / \alpha}, \infty) K \beta t^{-\beta-1} dt \right)  \label{DarstellungPsi}
\end{align}
for all $(s,y) \in \Rplus^2$, $y>0$.
%Desweiteren muss wegen nicht-Degeneriertheit $\omega$ ein endliches Moment der Ordnung $\alpha$ besitzen, das heißt
%\begin{eqnarray}
%M:=\intnu x^{\alpha} \omega(dx) < \infty. \label{TailMomentomega}
%\end{eqnarray}
\end{prop}
\begin{proof}
For the proof we look at the two additive parts of the L\'evy measure in \eqref{LevyMaDarst2} separately. For the first part we get
\begin{align*}
\Psi_1(s,y):=&\int_{\Rplus^2}\left(1-e^{-st}\ind{\left[0,y\right]}(x)\right)\epsilon_{0}(dt)C\alpha x^{-\alpha-1} dx\\
=&\int_{0}^{\infty} \ind{(y,\infty)} C \alpha x^{-\alpha -1} dx = Cy^{-\alpha}.
\end{align*}
For the second part we compute
\begin{align*}
\Psi_2(s,y):=&\int_{\Rplus^2}\left(1-e^{-st}\ind{\left[0,y\right]}(x)\right)(t^{\beta/\alpha} \omega) (dx) K \beta t^{-\beta -1} dt \\
=&\int_{0}^{\infty}\left(1-e^{-st}\right)K \beta t^{-\beta -1} dt+ \int_{\Rplus^2} e^{-st} \ind{(y,\infty)}(x) (t^{\beta/\alpha} \omega)(dx) K \beta t^{-\beta-1} dt\\
=&K\Gamma(1-\beta)s^{\beta}+\int_{0}^{\infty} e^{-st} \omega(t^{-\beta/\alpha}y,\infty) K \beta t^{-\beta-1} dt\\
=&K \Gamma(1-\beta)s^{\beta}+y^{-\alpha} \int_{0}^{\infty}e^{-s u y^{\alpha/\beta}}\omega(u^{-\beta/\alpha},\infty) K \beta u^{-\beta-1} du.
\end{align*}
The C-L exponent $\Psi$ of the limit distribution $P_{(D,A)}$ is $\Psi(s,y)=\Psi_1(s,y)+\Psi_2(s,y)$ and this corresponds to \eqref{DarstellungPsi}. 
\end{proof}

After analysing the $\alpha$-Fr\'echet case above, we now consider the general case, where $A$ in \eqref{AnnahmeEinleitung} can have any extreme value distribution. As before, let $x_0\in [-\infty,\infty)$ denote the left endpoint of $F_A$. Furthermore let $x_1$ denote the right endpoint of $F_A$.

\begin{theorem}\label{gencase}
Let $(W,J), (W_i,J_i)_{i\in\N}$ be i.i.d. $\rr_+\times\rr$ valued random vectors. Furthermore let $(D,A)$ be $\rr_+\times\rr$ valued with non-degenerate marginals. Then the following are equivalent:
\begin{enumerate}
\item[(a)] There exist sequences $(a_n),(b_n),(c_n)$ with $a_n,b_n>0$ and $c_n\in\rr$ such that
\begin{equation}\label{eqP1}
\bigl(a_nS(n),b_n\bigl(M(n)-c_n\bigr)\bigr) \WKonvergenz (D,A) ,
\end{equation}
that is $(W,J)\in \textnormal{sum-max-DOA}(D,A)$.
\item[(b)] There exist sequences $(a_n),(b_n),(c_n)$ with $a_n,b_n>0$ and $c_n\in\rr$ such that
\begin{equation}\label{eqP2}
n\cdot P_{(a_nW,b_n(J-c_n))}\vKonvergenzninfty \eta, 
\end{equation}
where $\eta$ is a L\'evy measure on $(\Raum,\plusmax)$.
\end{enumerate}
Then $(D,A)$ is sum-max stable and has L\'evy representation $[x_0,0,\eta]$.
\end{theorem}

\begin{proof}
The proof is similar to the proof of Theorem \ref{Theoremjointsum/max} and left to the reader.
\end{proof}

As in the $\alpha$-Fr\'echet case it is also possible to describe the L\'evy measure $\eta$ in \eqref{eqP2} in the general case.

\begin{theorem}\label{levymeasgen}
Under the assumptions of Theorem \ref{gencase}, there exist constants $C\geq 0, K>0$ and a probability measure $\omega\in\mathcal M^1(\rr)$ with $\omega(\rr_+)>0$ and $\int_0^\infty x\,\omega(dx)<\infty$ such that the L\'evy measure of $(D,A)$ on $\rr_+\times [x_0,x_1]\setminus\{(0,x_0)\}$ is given by
\begin{equation}\label{eqoooo}
\eta(dt,dx)=\varepsilon_0(dt)C\Gamma(x)^{-2}\Gamma'(x)dx+ \ind{(0,\infty) \times (x_0,x_1)} (t,x)\bigl(\Gamma^{-1}(t^\beta\omega)\bigr)(dx)\ K\beta t^{-\beta-1}dt ,
\end{equation}
where $\Gamma(x)=1/(-\log F_A(x))$.
\end{theorem}

\begin{proof}
Observe that $(D,\Gamma(A))$ is sum-max stable where $\Gamma(A)$ is $1$-Fr\'echet. In view of Theorem \ref{TheoremLevyMaDarst}  the L\'evy measure $\tilde\eta$ of $(D,\Gamma(A))$ has the representation
\begin{equation}\label{eqss1}
\tilde\eta(dt,dx)=\varepsilon_0(dt)Cx^{-2}dx + \ind{(0,\infty) \times \rr_+} (t,x)\bigl(t^\beta\omega\bigr)(dx)\ K\beta t^{-\beta-1}dt
\end{equation}
with constants $C\geq 0,K>0$ and $\omega\in\mathcal M^1(\rr)$ with $\omega(\rr_+)>0$ and $\int_0^\infty x\,\omega(dx)<\infty$. Now let $\tilde\Psi$ denote the C-L-exponent of $(D,\Gamma(A))$. Since
\[ \mathcal L\bigl(P_{(D,A)}\bigr)(s,y)=\mathcal L\bigl(P_{(D,\Gamma(A))}\bigr)(s,\Gamma(y))=\exp\bigl(-\tilde\Psi(s,\Gamma(y))\bigr) ,\]
the C-L-exponent of $(D,A)$ is given by $\Psi(s,y)=\tilde\Psi(s,\Gamma(y))$. Setting $g(t,x)=(t,\Gamma^{-1}(x))$ we therefore get
\begin{equation*}
\begin{split}
\Psi(s,y) &= \tilde\Psi(s,\Gamma(y)) \\
&=\int_{\rr_+\times \rr_+}\bigl(1-e^{-st}\ind{[-\infty,\Gamma(y)]}(x)\bigr)\tilde\eta(dt,dx) \\
&=\int_{\rr_+\times\rr_+}\bigl(1-e^{-st}\ind{([-\infty,y]}(\Gamma^{-1}(x))\bigr)\tilde\eta(dt,dx) \\
&=\int_{\rr_+\times [x_0,x_1)}\bigl(1-e^{-st}\ind{[-\infty,y]}(x)\bigr)\, g(\tilde\eta)(dt,dx) ,
\end{split}
\end{equation*}
so $g(\tilde\eta)$ is the L\'evy measure of $(D,A)$. Using \eqref{eqss1} it is easy to see that $g(\tilde\eta)$ has the form \eqref{eqoooo} and the proof is complete.
\end{proof}

\section{Examples}
In this section we present some examples of random vectors $(W,J)$ which are in the domain of attraction of a sum-max stable distribution and calculate the L\'evy measures of the corresponding limit distributions as well as the C-L exponent, using the theory developed in section 3 above. In the following let $(W_i,J_i)_{i \in \mathbb{N}}$ be a sequence of $\Rplus \times \mathbb{R}$-valued random vectors with $(W_i,J_i)\stackrel{d}{=}(W,J)$.

\begin{example}\label{W=JBSP}
First we  consider the case of complete dependence, that is $W_i=J_i$ for all $i \in \mathbb{N}$. This is the case which was already studied in \cite{chow1978sum}. We choose $W$ to be in the strict normal domain of attraction (meaning that we have $a_n=n^{-1/\beta}$ in \eqref{AnnahmeEinleitung}) of a $\beta$-stable random variable $D$ with $0<\beta<1$ and $E\left(e^{-sD}\right)=\exp(-s^{\beta})$. The L\'evy measure of $P_D$ is given by
\begin{align}\label{lm1}
\Phi_{\beta}(dt)=\eta(dt,\Rplus)=\frac{\beta}{\Gamma(1-\beta)}t^{-\beta-1}dt.
\end{align} 
We now choose $a_n=b_n=n^{-1/\beta}$ and $\alpha=\beta$ to get
\begin{align*}
n \cdot P \Big(n^{-1/\beta} W >t, n^{-1/\beta} J > y\Big)&=n\cdot P \left(n^{-1/\beta}W > \max(t,y) \right)\\
&\Konvergenzninfty \frac{1}{\Gamma(1-\beta)}\max(t,y)^{-\beta}
\end{align*}
for $t,y>0$. Thus we know with Theorem \ref{Theoremjointsum/max} that the L\'evy measure $\eta$ is given by $\eta((t,\infty) \times (y,\infty))=\frac{1}{\Gamma(1-\beta)}\max(t,y)^{-\beta}$.
If we choose $\alpha=\beta,\ \omega=\epsilon_{1},\ K=\frac{1}{\Gamma(1-\beta)} \text{ and } C=0$  in equation \eqref{LevyMaDarst2} we as well get
\begin{align*}
&\eta((t,\infty) \times (y,\infty))=\int_{t}^{\infty}\int_{y}^{\infty} (r\epsilon_1) (dx) \frac{\beta}{\Gamma(1-\beta)}r^{-\beta-1}dr\\
&=\int_{0}^{\infty} \ind{(t,\infty)}(r) \ind{(y,\infty)}(r)\frac{\beta}{\Gamma(1-\beta)} r^{-\beta-1} dr=\frac{1}{\Gamma(1-\beta)}\max(t,y)^{-\beta}.
\end{align*}
Hence the limit distribution  in  case of total dependence is uniquely determined by $P_{(D,A)} \sim [0,0,\eta]$ with
\begin{align*}
\eta(dt,dx)=\ind{(0,\infty) \times \Rplus} \epsilon_{t}(dx) \Phi_{\beta} (dt).
\end{align*}
Setting $\alpha=\beta,\ \omega=\epsilon_{1},\ K=\frac{1}{\Gamma(1-\beta)} \text{ and } C=0$ in \eqref{DarstellungPsi}, the C-L exponent in this case is given by
\begin{align}\label{BSP1Gl1}
\Psi(s,y)=s^{\beta} + y ^{-\beta} \left( \int_1^{\infty} e^{-sty}  \frac{\beta}{\Gamma(1-\beta)} t^{-\beta -1} dt \right).
\end{align}
\end{example}

\begin{example}\label{BSPNormalvertGemKonv} \ \\
Again we choose $W$ to be in the strict normal domain of attraction of a $\beta$-stable random variable $D$ with $0<\beta<1$ and $E\left(e^{-sD}\right)=\exp(-s^{\beta})$. Furthermore let $Z$ be a standard normal distributed random variable, i.e. $Z \sim\mathcal{N}_{0,1}$ and $Z$ is independent of $W$. We define $J:=W^{1/2}Z$, hence the conditional distribution of $J$ given $W=t$ is $\mathcal{N}_{0,t}$ distributed. Define a homeomorphism $T: \Rplus \times \mathbb{R} \rightarrow \Rplus \times \mathbb{R}$ with $T(t,x)=(t,t^{1/2}x)$. Then we get for continuity sets $A \subseteq \Rplushochzwei$ that are bounded away from $\left\{(0,0)\right\}$
\begin{align*}
n \cdot P_{(n^{-1/\beta}W,n^{-1/2\beta}J)}(A)&=n \cdot P_{T(n^{-1/\beta}W,Z)}(A)\\
&=n \cdot \left(P_{n^{-1/\beta}W} \otimes P_{Z} \right) (T^{-1}(A))\\
&\Konvergenzninfty (\Phi_{\beta} \otimes \mathcal{N}_{0,1}) (T^{-1}(A)),
\end{align*}
where $\Phi_{\beta}$ is again the L\'evy measure of $D$, given by \eqref{lm1}. Hence the L\'evy measure of $(D,A)$ is  given by
\begin{align*}
\eta(dt,dx)=T\left(\Phi_{\beta} \otimes \mathcal{N}_{0,1} \right) (dt,dx)=\mathcal{N}_{0,t}(dx)\Phi_{\beta}(dt).
\end{align*}
This coincides with \eqref{LevyMaDarst2} in Theorem \ref{TheoremLevyMaDarst}, if we choose $C=0, \alpha=2\beta, \omega=\mathcal{N}_{0,1}$ and $K=\frac{1}{\Gamma(1-\beta)}$. For the C-L exponent we get with \eqref{DarstellungPsi} in Proposition \ref{LaplaceExponent}
\begin{align}\label{BSP2Gl1}
\Psi(s,y)= s^{\beta}+y^{-2\beta} \int_{0}^{\infty} e^{-sty^2} \mathcal{N}_{0,t}(1,\infty) \frac{\beta}{\Gamma(1-\beta)}t^{-\beta-1} dt.
\end{align}
\end{example}

\begin{example}\label{BSPgemeinsame3}\ \\
Again we choose $W$ to be in the strict normal domain of attraction of a $\beta$-stable random variable $D$ with $0<\beta<1$ and $E\left(e^{-sD}\right)=\exp(-s^{\beta})$. Furthermore let $Z$ be a $\gamma-$Fr\'echet distributed random variable with distribution function ${P(Z \leq t)=e^{-C_{1}t^{-\gamma}}}$ with $C_{1}>0$ and $\gamma>0$, and $Z$ is independent of $W$. We define ${J:=W^{1/\gamma}Z}$.
Let $T: \Rplus \times \mathbb{R} \rightarrow \Rplus \times \mathbb{R}$ be the homeomorphism with $T(t,x)=(t,t^{1/\gamma} x)$. We then have for all continuity sets $B \subseteq \Rplushochzwei$ bounded away from $\left\{(0,0)\right\}$
\begin{align*}
n \cdot P_{(n^{-1/\beta}W, n^{-\frac{1}{\beta \gamma}}J)} (B)&=n \cdot P_{T(n^{-1/\beta}W,Z)}(B) \nonumber \\
&=n \cdot P_{(n^{-1/\beta}W,Z)}(T^{-1}(B))\\
&=n \cdot \left( P_{n^{-1/\beta}W} \otimes P_{Z} \right) (T^{-1}(B)) \nonumber \\
& \Konvergenzninfty \left(\Phi_{\beta} \otimes P_{Z}\right) (T^{-1}(B))=T(\Phi_{\beta} \otimes P_{Z}) (B),
\end{align*}
where $\Phi_{\beta}$ denotes the L\'evy measure of $P_D$. Consequently the L\'evy measure of $(D,A)$ is given by
\begin{align*}
\eta(dt,dx)=\bigl(t^{1/\gamma} P_{Z}\bigr)(dx) \frac{\beta}{\Gamma(1-\beta)}t^{-\beta-1}dt.
\end{align*}
This coincides with Theorem \ref{TheoremLevyMaDarst} if we let $\omega=P_{Z}, \alpha=\beta\gamma$, $K=1/\Gamma(1-\beta)$ and $C=0$ in \eqref{LevyMaDarst2}. With Theorem \ref{Theoremjointsum/max} we know
\begin{align*}
(n^{-1/\beta} S(n), n^{-1/\beta \gamma} M(n)) \xLongrightarrow[n \rightarrow \infty]{} (D,A),
\end{align*}
where $D$ strictly stable with $0<\beta<1$ and $A$ is $\alpha=\beta \gamma$-Fr\'echet distributed. The condition $\int_{0}^{\infty} x^{\alpha} \omega(dx)<\infty$ is fulfilled then due to $0<\beta<1$ is $\alpha=\beta \gamma < \gamma$ and $\omega$ is $\gamma$-Fr\'echet distributed. This means that $(W,J)$ is in the sum-max domain of attraction of $(D,A)$. With Proposition \ref{LaplaceExponent} we compute the C-L exponent with $K:=1/\Gamma(1-\beta)$:
\begin{align*}
\begin{split}
\Psi(s,y)&=s^{\beta}+y^{-\beta \gamma} \int_{0}^{\infty} e^{-sty^{\gamma} } \omega(t^{-1/\gamma},\infty) \beta K t^{-\beta-1} dt\\
&=s^{\beta}+y^{-\beta \gamma} \int_{0}^{\infty}e^{-sty^{\gamma}} (1-e^{-C_{1}t}) \beta Kt^{-\beta-1} dt\\
&=s^{\beta}+y^{-\beta \gamma}\left(\int_{0}^{\infty}(e^{-sty^{\gamma}}-1)\beta Kt^{-\beta-1} dt+\int_{0}^{\infty}(1-e^{-t(sy^\gamma+C_{1})})\beta Kt^{-\beta-1} dt\right)\\
&=s^{\beta}+y^{-\beta \gamma}\left(-(sy^{\gamma})^{\beta}+(sy^{\gamma}+C_1)^{\beta}\right)\\
&=y^{-\beta \gamma}(sy^{\gamma}+C_1)^{\beta}\\
&=(s+C_{1}y^{-\gamma})^{\beta}.
\end{split}
\end{align*}
\end{example}

\section{Appendix: Proofs}
In this section we give some of the technical proofs of section 2 above.
\begin{proof}[Proof of Lemma \ref{lemmaCL}]
First we assume that $\varphi$ is the C-L transform of a probability measure $\mu \in \mathcal{M}^1(\Rplus \times \Rpmu)$. The map
\[h:(\Rplus \times \Rpmu,\plusmax) \rightarrow \boldsymbol{\hat{S}}, \ (t,x) \rightarrow e^{-t \, \LargerCdot \,} \ind{[-\infty,\, \LargerCdot \, ]}(x)\] is an injective homomorphism, where  $\bf{\hat{S}}$ is the set of all bounded semicharacters on $S=(\Rplus \times \Rpmu, \plusmin)$ in \eqref{MengeallerHG}. We get
\begin{align*}
\varphi(s,y)&=\int_{\Rplus}\int_{\Rpmu} e^{-st}\ind{[-\infty,y]}(x) \mu(dt,dx)\\
&=\int_{\hat{\textbf{S}}}\rho(s,y)h(\mu)(d\rho), \text{ for all } (s,y) \in \Rplus \times \Rkomp.
\end{align*}
Theorem 4.2.5 in \cite{Berg} implies that $\varphi$ is positive semidefinite. It is obvious that $\varphi$ is also bounded and normalized. If we put $s=0$ we get \[\varphi(0,y)=\mu(\Rplus \times [-\infty,y])\] for all $y \in \Rkomp$ and hence the distribution function of a probability measure on $\Rpmu$. Otherwise, if we put $y=\infty$ we get \[\varphi(s,\infty)=\int_{\Rplus}e^{-st}\mu(dt,\Rpmu)\] for all $s \in \Rplus$, hence the Laplace transform of a probability measure on $\Rplus$.\par \noindent
Conversely $\varphi$ is now a positive semidefinit, bounded and normalized function on $(\Rplus \times \Rkomp,\plusmin)$. Theorem 4.2.8 in \cite{Berg} implies, that there exists exactly one probability measure $\mu$ on the set of bounded semicharacters $\bf{\hat{S}}$ of the semigroup $S=(\Rplus \times \Rpmu,\plusmin)$, such that
\begin{align*}
\varphi(s,y)=\int_{\hat{\textbf{S}}}\rho(s,y)\mu(d\rho), \text{ for all } (s,y) \in \Rplus \times \Rkomp.
\end{align*}
We divide $\bf{\hat{S}}$ in \eqref{MengeallerHG} in the two disjoint subsets
\begin{align*}
\bf{\hat{S}^{'}}&=\Big\{e^{-t\LargerCdot}\ind{[x,\infty]}(\LargerCdot), x \in [-\infty,\infty], t \in [0,\infty]\Big\} \text{ and } \\ 
\bf{\hat{S}^{''}}&=\Big\{e^{-t\LargerCdot}\ind{(x,\infty]}(\LargerCdot), x \in [-\infty,\infty), t \in [0,\infty]\Big\}.
\end{align*} 
We define the isomorphisms $h_1:\Rplus \times \Rpmu \rightarrow  \bf{\hat{S}^{'}}$ and $h_2 :\Rplus \times \Ru \rightarrow  \bf{\hat{S}^{''}}$ by \[h_1(t,x):=e^{-t \, \LargerCdot \,}\ind{[x,\infty]}(\LargerCdot) \hspace{0.5cm} \text{and} \hspace{0.5cm} h_2(t,x):=e^{-t \, \LargerCdot \,}\ind{(x,\infty]}(\LargerCdot).\] Hence we get
\begin{align*}
\varphi(s,y)&=\int_{[0,\infty]}\int_{\Rpmu}e^{-st}\ind{[-\infty,y]}(x)h_1^{-1}(\mu)(dt,dx)+\int_{[0,\infty]}\int_{\Ru}e^{-st}\ind{[-\infty,y)}(x)h_2^{-1}(\mu)(dt,dx)\\
&=\int_{\Rpmu} \ind{[-\infty,y]}(x)\left\{\int_{[0,\infty)}e^{-st}h_1^{-1}(\mu)(dt,dx) + h_1^{-1}(\mu)(\left\{\infty\right\},dx) \cdot \ind{\left\{0\right\}}(s)\right\}\\
&\hspace{5mm}+\int_{\Ru} \ind{[-\infty,y)}(x)\left\{\int_{[0,\infty)}e^{-st}h_2^{-1}(\mu)(dt,dx) +h_2^{-1}(\mu)(\left\{\infty\right\},dx) \cdot \ind{\left\{0\right\}}(s)\right\}.
\end{align*}
Due to the right continuity of $\varphi(0,y)$ in $y \in \Rkomp$ there are only  two possible cases: Either $h^{-1}_2(\mu)([0,\infty]\times\left\{y\right\})=0$ for all $y \in \Rkomp$ or  $h^{-1}_2(\mu)([0,\infty] \times \, \LargerCdot) \equiv 0$. In the first case we choose $\tilde{\mu}:=h^{-1}_1(\mu)+h^{-1}_2(\mu)$. In the second case the last integral disappears and we choose $\tilde{\mu}:=h^{-1}_1(\mu)$. Since $\varphi(s,\infty)$ is continuous in $s$ it follows that $h_i^{-1}(\mu)(\left\{\infty\right\},dx)=0$ for $i=1,2$. Due to the fact that $\varphi$ is normalized, $\mu$ is a probability measure. Hence we get the desired form in \eqref{C-L-Transformierte}.
\end{proof}

\begin{proof}[Proof of Lemma \ref{gleichmKonv}]
We write
\begin{align*}
\mathcal{L}(\mu_n)(s_n,y_n)&=\int_{\Raum} e^{-s_nt} \ind{[-\infty,y_n]}(x) \mu_n(dt,dx)\\
&=\int_{\mathbb{R}_+} e^{-s_nt} \mu_n(dt,[-\infty,y_n])\\
&=\textnormal{L}(\tilde{\mu}_n)(s_n),
\end{align*}
where we define the measures $\tilde{\mu}_n(dt):=\mu_n(dt,[-\infty,y_n])$ and $\tilde{\mu}:=\mu(dt,[-\infty,y])$. The assertion follows, if we can show that
\begin{align}
\tilde{\mu}_n \wKonvergenzninfty \tilde{\mu}. \label{Konvergenzmutilde}
\end{align}
Then due to the uniform convergence of the Laplace transform it follows 
\begin{align*}
\textnormal{L}(\tilde{\mu_n}) (s_n) \xrightarrow[n \rightarrow \infty]{} \textnormal{L}(\tilde{\mu})(s).
\end{align*}
So it remains to show \eqref{Konvergenzmutilde}. Because $\mu_n \xlongrightarrow[]{w}\mu$, %we have that $\mu_n(A) \xrightarrow[]{} \mu(A)$ for $n \rightarrow \infty$ for all $A\in \mathcal{B}(\Raum)$ with $\mu(\partial A)=0$. Especially
we know that
\begin{align}
\mu_n(A_1 \times A_2) \xrightarrow[n \rightarrow \infty]{} \mu(A_1 \times A_2) \label{Kompglm1}
\end{align}
for $ A_1 \times A_2 \in \mathcal{B}(\Rplus \times \Rpmu)$ with $\mu(\partial(A_1 \times A_2))=0$. Hence
\begin{align}
\mu_n(B \times [-\infty,y]) \xrightarrow[n \rightarrow \infty]{} \mu(B \times [-\infty,y]) \label{Kompglm2}
\end{align}
for all $B \in \mathcal{B}(\Rplus)$ with $\mu(\partial B \times [-\infty,y])=0$, then \eqref{Kompglm2} is fulfilled for all sets $B \times [-\infty,y]$ with $\mu(\partial(B \times [-\infty,y]))=0$. It is \[\partial(B \times [-\infty,y])=\partial B \times [-\infty,y] \cup B \times \left\{y\right\}\] and because $y$ is a point of continuity of the function $\mu(\Rplus \times [-\infty,y])$, it follows from $\mu(\partial B \times [-\infty,y])=0$ that $\mu(\partial(B \times [-\infty,y]))=0$. For a set $B \in \mathcal{B}(\Rplus)$, $\mu_n(B \times [-\infty,y])$ is an increasing, right continuous function which is continuous in $y$, and so an (improper) distribution function. But then it follows that \[\mu_n(B \times [-\infty,y_n]) \xlongrightarrow[n \rightarrow \infty]{} \mu(B \times [-\infty,y])\]
if $y_n \longrightarrow y$ for $n\rightarrow \infty$ an \eqref{Konvergenzmutilde} holds true.
\end{proof}

\begin{proof}[Proof of Lemma \ref{lemmawGrenzunendteil}]
From Theorem \ref{Stetigkeitssatz} we know that $\mathcal{L}(\mu_n)(s,y) \xlongrightarrow[]{} \mathcal{L}(\mu)(s,y)$ as $n \rightarrow \infty$ in all $(s,y) \in \Rplus \times \Rkomp$ but countable many $y\in \Rkomp$. Since the probability measures $\mu_n$ are $\plusmax$-infinite divisible there exists a measure $\mu_{m,n}$ for all $n,m \geq 1$ such that $\mu_n=\mu^{\circledast m}_{m,n}$. Because of Proposition \ref{EigenschaftenC-L} (b) and (c) this is equivalent to \[\mathcal{L}(\mu_n)(s,y)=\mathcal{L}(\mu_{m,n})^m(s,y)\] for all $(s,y) \in \Rplus \times \Rkomp$. It then follows that
\[\mathcal{L}(\mu_{m,n})(s,y)=\mathcal{L}(\mu_{n})^{1/m}(s,y) \xlongrightarrow[n \rightarrow \infty]{}  (\mathcal{L}(\mu))^{1/m}(s,y)\] 
in all $(s,y) \in \Rplus \times \Rkomp$ but countable many $y \in \Rkomp$. Since
\begin{align*}
\lim_{s \downarrow 0} \left(\mathcal{L}(\mu)\right)^{1/m}(s,\infty)=\left(\lim_{s \downarrow 0}  \int_{0}^{\infty} e^{-st} \mu(dt,\Rpmu)\right)^{1/m}=(\mu(\Raum))^{1/m}=1 
\end{align*}
it follows from Theorem \ref{Stetigkeitssatz} that there exists a measure $\nu \in \mathcal{M}^1(\Raum)$ with $\mathcal{L}(\nu)=\left(\mathcal{L}(\mu)\right)^{1/m}$.
Hence $\mu=\nu^{\circledast m}$, so $\mu$ is $\plusmax$-infinite divisible. 
\end{proof}

\begin{proof}[Proof of Lemma \ref{Konvergenzsatz1}]
In view of Lemma \ref{lemmawGrenzunendteil} we already know that $\mu$ is $\plusmax$-infinite divisible. By Theorem \ref{TheoremLKD} an Remark \ref{LKDkorollar} we know that the C-L exponent has the form 
\begin{align*}
\Psi(s,y)=a \cdot s +\int_{\Rplus}\int_{\Rpmu}\left(1-e^{-st}\cdot\ind{\left[-\infty,y\right]}(x)\right)\eta(dt,dx) \ \forall (s,y) \in \Rplus \times (x_0,\infty].
\end{align*}
First we define for any $h_1>0$ and any $h_2>x_0$ a function $\Psi^{*}$ on $\Rplus \times (x_0,\infty]$ by
\begin{align*}
\Psi^{*}(s,y)=\Psi(s+h_1,y \wedge h_2)-\Psi(s,y).
\end{align*}
For $\Psi^{*}$ we get for all $(s,y) \in \Rplus \times (x_0,\infty]$
\begin{align*}
\Psi^{*}(s,y)&=ah_1+\int_{\Rplus \times \Rpmu}(e^{-st}\ind{[-\infty,y]}(x)-e^{-(s+h_1)t}\ind{[-\infty,y \wedge h_2]}(x))\eta(dt,dx)\\
&=ah_1+\int_{\Rplus \times \Rpmu}e^{-st}\ind{[-\infty,y]}(x)(1-e^{-h_1t}\ind{[-\infty,h_2]}(x))\eta(dt,dx)\\
&=ah_1+\int_{\Rplus \times \Rpmu}e^{-st}\ind{[-\infty,y]}(x)K(t,x)\eta(dt,dx),
\end{align*}
where $K(t,x):=1-e^{-h_1t}\ind{[-\infty,h_2]}(x)$. By Taylor expansion we get for all $x \leq h_2$
\begin{align*}
K(t,x)=h_{1}t+o (t) \text{ as } t\rightarrow 0. 
\end{align*}
Now we define a measure $M$ on $\Rplus \times [x_0,\infty]$ by
\begin{align*}
&M(dt,dx):=K(t,x)\eta(dt,dx) \text{ on } \Rplus \times [x_0,\infty]\backslash \left\{(0,x_0)\right\}\\
&\hspace{5mm}\text{and }\\ 
&M(\left\{(0,x_0)\right\}):=ah_1.
\end{align*}
This is a finite measure, because for $0<\epsilon<1$ we get
\begin{align*}
&M(\Rplus \times [x_0,\infty])\\
&=ah_1+\int_{t \geq \epsilon}\int_{x_0 \leq x \leq h_2}(1-e^{-h_1t})\eta(dt,dx)+\int_{0<t<\epsilon}\int_{x_0 \leq x \leq h_2}(1-e^{-h_1t})\eta(dt,dx)\\
&\hspace{5mm}+\int_{t \geq \epsilon}\int_{x>h_2} 1 \ \eta(dt,dx)+\int_{0<t<\epsilon}\int_{x>h_2} 1 \ \eta(dt,dx) < \infty.
\end{align*}
The first integral is finite, because $\eta([\epsilon,\infty] \times [x_0,\infty]) < \infty$. The second integral is finite because of the Taylor expansion of $K$ and due to the integrability condition of $\eta$ we have that $\int_{t < \epsilon} t \ \eta(dt,[x_0,\infty]) <\infty$. The third integral is finite, since we have $\eta([\epsilon,\infty) \times (h_2,\infty])<\infty$ and the last integral due to $\eta(\Rplus \times (h_2,\infty])<\infty.$
Hence the function $\Psi^{*}$ is the C-L transform of the finite measure $M$ and therefore, in view of  Proposition \ref{EigenschaftenC-L} (c), $M$ is uniquely determined   by $\Psi^{*}$ and hence also by $\Psi$. \par
\ \\
Now we show (b). The measures $\mu_n \sim [x_n,a_n,\eta_n]$ are {$\plusmax$-infinite} divisible measures with $\mu_n \wKonvergenz \mu$ as $n \rightarrow \infty$ and ${\mu \sim [x_0,a,\eta]}$. Define $\Psi_{n}^{*}$ and $M_n$ as $\Psi^{*}$ and $M$ above,that is
\begin{align}
&M_n(dt,dx)=K(t,x) \ \eta_n(dt,dx) \\
&\hspace{5mm}\text{and } \nonumber \\
&\Psi_n^{*}(s,y)=a_nh_1+\int_{\Raum} e^{-st} \ind{[-\infty,y]}(x)\left(1-e^{-h_1 t}\ind{[-\infty,h_2]}(x)\right) \eta_n(dt,dx),
\end{align}
where $n$ is in the case $x_n \xrightarrow[]{} x_0$ if chosen large enough to ensure $x_n \leq h_2$.
In view of  the Continuity Theorem it follows that $\Psi_n(s,y) \rightarrow \Psi(s,y)$ in all $(s,y)\in \Rplus \times \Rkomp$ but countable many $y>x_0$ and hence also $\Psi_n^{*}(s,y) \rightarrow \Psi^{*}(s,y)$. Because $\Psi^{*}$ is the C-L transform of $M$, we get with Theorem \ref{Stetigkeitssatz}, that \[M_n \xlongrightarrow[n \rightarrow \infty]{w} M.\]\\
Now we choose $ S \in \mathcal{B}(\Rplus \times [h_2,\infty])$ with $\text{dist}((0,h_2),S)>0$ and $\eta(\partial S )=0$. Then is $M(\partial S)=0$ as well and it follows \[M_n(S)=\int_{S}K(t,x)  \eta_n (dt,dx ) \xrightarrow[n \rightarrow \infty]{} M(S)=\int_{S}K(t,x) \eta(dt,dx).\] Due to $K(t,x)>0$ if $t > 0$ or $x>h_2$ it follows that $\eta_n(S) \xrightarrow[]{} \eta(S)$ as $n \rightarrow \infty$. Since $h_2>x_0$ was chosen arbitrarily, it follows ${\eta_{n}} \xrightarrow[]{v'} \eta$, i.e. (b) is fulfilled. Because ${\mu_n \xrightarrow[]{w} \mu}$ implies $\Psi_n(s,y) \rightarrow \Psi(s,y)$ as $n \rightarrow \infty$ for all $(s,y)\in \mathbb{R}_{+} \times \Rkomp$ where $\Psi$ is continuous, we have $\Psi_n(s,\infty) \xrightarrow[]{} \Psi(s,\infty)$
and it follows, that $a_n \xrightarrow[]{} a$ as $n \rightarrow \infty$. Hence (a) is also fulfilled. It remains to show (c). For all but countable many $y>x_0$ we have
\begin{align*}
\Psi(s,y)&=\limn \left[a_ns + \int_{\Rplus \times \Rpmu}(1-e^{-st}\ind{[-\infty,y]}(x))\eta_n(dt,dx)\right]\\
&=as+\overbar{\limn} \left[\int_{\Rplus \times \Rpmu}(1-e^{-st})\eta_n(dt,dx)+\int_{\Rplus \times \Rpmu}e^{-st}\ind{(y,\infty)}(x)\eta_n(dt,dx)\right].
\end{align*}
We divide the first integral into two parts and get
\begin{align*}
\Psi(s,y)&=as+\lim_{\epsilon \downarrow 0}\overbar{\limn}\left[\int_{\left\{t:0 \leq t<\epsilon \right\}}(1-e^{-st})\eta_n(dt,\Rpmu)+\int_{\left\{t: t \geq \epsilon\right\}}(1-e^{-st})\eta_n(dt,\Rpmu)\right]\\
&\hspace{5mm}+\int_{\Rplus \times \Rpmu}e^{-st}\ind{(y,\infty)}(x)\eta(dt,dx)\\
&=as+\lim_{\epsilon \downarrow 0}\overbar{\limn}\left[\int_{\left\{t:0 \leq t<\epsilon \right\}}(1-e^{-st})\eta_n(dt,\Rpmu)\right]+\int_{\Rplus}(1-e^{-st})\eta(dt,\Rpmu)\\
&\hspace{5mm}+\int_{\Rplus \times \Rpmu}e^{-st}\ind{(y,\infty)}(x)\eta(dt,dx)\\
&=s\left(a+\lim_{\epsilon \downarrow 0}\overbar{\limn}\left[\int_{\left\{t:0 \leq t<\epsilon \right\}}t\ \eta_n(dt,\Rpmu)\right]\right)+\int_{\Rplus \times\Rpmu}(1-e^{-st}\ind{[-\infty,y]}(x))\eta(dt,dx).
\end{align*}
Hence (c)  also holds.\par
\ \\
Conversely we assume that (a)-(c) is fulfilled. It then follows for all $y>x_0$ with the same decomposition:
\begin{align*}
\limn \Psi_n(s,y)
&=as+ \lim_{\epsilon \downarrow 0} \overbar{\limn}\bigg[\int_{\left\{t:0<t<\epsilon \right\}} st \ \eta_n(dt,\Rpmu)\bigg]\\
&\hspace{5mm}+\int_{\Rplus \times \Rpmu}(1-e^{-st}\ind{[-\infty,y]}(x))\eta(dt,dx)\\
&=\Psi(s,y).
%&=a \cdot s +\int_{\Rplus \times [x_0,\infty]}(1-e^{-st}\cdot \ind{[x_0,y]}(x))\eta(dt,dx).
\end{align*}
Hence\[\mathcal{L}(\mu_n)(s,y)=e^{-\Psi_n(s,y)} \xrightarrow[n \rightarrow \infty]{}e^{-\Psi(s,y)}=\mathcal{L}(\mu)(s,y)\] for at most countable many $y>x_0$ and as a consequence it follows with Theorem \ref{Stetigkeitssatz} \[\mu_n \xrightarrow[n \rightarrow \infty]{w} \mu\] with $\mu \sim [x_0,a,\eta]$.
\end{proof}

\begin{proof}[Proof of Lemma \ref{LemmabegleitendeVerteilung}]
This probability measure is the analogue to the  compound Poisson distribution induced by the  convolution $\circledast$ which itself is induced by the semigroup operation $\plusmax$. Hence it follows similar as for the usual Poisson distribution (see \cite[Corollary 3.1.8.]{thebook}) that the C-L transform is given by $\exp(-c\left[1-\mathcal{L}(\mu)(s,y)\right])$ for all $y\geq x_0$:
\begin{align*}
\mathcal{L}(\Pi(c,\mu))(s,y)&=e^{-c} \sum_{k=0}^{\infty} \frac{c^{k}}{k!}(\mathcal{L}(\mu)(s,y))^{k}\\
&=\exp\left(-c (1-\mathcal{L}(\mu)(s,y))\right) \text{ for all } y\geq x_0.
\end{align*}
For the C-L exponent it then follows that
\begin{align*}
\Psi(s,y)&=c (1-\mathcal{L}(\mu)(s,y))\\
&=\int_{\Raum} (1-e^{-st}\ind{[-\infty,y]}(x)) (c \cdot \mu)(dt,dx).
\end{align*}
The measure $\Pi(c,\mu)$ has the same left endpoint as the measure $\mu$ and since $\mu^{\circledast 0}=\varepsilon_{(0,x_0)}$, it is obvious that $\mathcal{L}(\Pi(c,\mu))(s,y)>0$ for all $(s,y) \in \Rplus \times [x_0,\infty]$.  Hence, by Theorem \ref{TheoremLKD} the measure $\Pi(c,\mu)$ is $\plusmax$-infinite divisible with L\'evy representation $\Pi(c,\mu) \sim \left[x_0,0, c \cdot \mu \right]$.
\end{proof}

\begin{proof}[Proof of Lemma \ref{LemmabegleitendeVerteilung2}]
This also follows along the same lines as for to  for the compound Poisson distribution. With Theorem \ref{Stetigkeitssatz} and Proposition \ref{EigenschaftenC-L} (b)
$\mu_n^{\circledast n} \wKonvergenz \nu$ as $n \rightarrow \infty$ if and only if \[\mathcal{L}(\mu_n)^n(s,y) \rightarrow \mathcal{L}(\nu)(s,y) \text{ as } n \rightarrow \infty\] in all $(s,y) \in \Rplus \times \Rkomp$ but countable many $y>x_0$.
Since $\mathcal{L}(\nu)(s,y)>0$ for all $y>x_0$, this is equivalent to 
\[n \cdot \log \mathcal{L}(\mu_n)(s,y) \xrightarrow[]{} \log \mathcal{L}(\nu)(s,y) \text{ as } n \rightarrow \infty\] in all but countable many $y>x_0$. Because $\mathcal{L}(\mu_n)(s,y) \rightarrow 1$ as $n \rightarrow \infty$ for all $y>x_0$ and $\log z \sim z-1$ as $z \rightarrow 1$, this is equivalent to
\[n \cdot (\mathcal{L}(\mu_{n})(s,y)-1) \xrightarrow[]{} \log \mathcal{L}(\nu)(s,y) \text{ as } n \rightarrow \infty\] in all but countable many $y>x_0$.
And this is equivalent to
\[\exp ( n\left[\mathcal{L}(\mu_n)(s,y)-1\right]) \xrightarrow[]{} \mathcal{L}(\nu)(s,y) \text{ as } n \rightarrow \infty\] in all but countable many $y>x_0$ and because of 
Theorem \ref{Stetigkeitssatz} and Lemma \ref{LemmabegleitendeVerteilung}, this is equivalent to \[\Pi(n,\mu_n) \xrightarrow[]{w} \nu \text{ as } n \rightarrow \infty.\]
\end{proof}

\end{document}